\documentclass[a4paper,11pt]{article}
\usepackage[margin=1in]{geometry}
\usepackage[T1]{fontenc}
\usepackage{amsthm,amsfonts,amssymb,amsmath}
\usepackage{microtype}
\usepackage{graphicx}
\usepackage{xcolor}
\usepackage{enumitem}
\usepackage[nocompress]{cite}

\usepackage[colorlinks=true,urlcolor=blue,linkcolor=blue,citecolor=blue]{hyperref}

\newtheorem{Theorem}{Theorem}
\newtheorem{Lemma}{Lemma}
	
\newtheorem{Corollary}{Corollary}	
\theoremstyle{definition}
\newtheorem{Remark}{Remark}	 
 
\newtheorem{Assumption}{Assumption} 

\numberwithin{equation}{section}

\newcommand{\C}{\mathbb{C}} 
\newcommand{\R}{\mathbb{R}} 
\newcommand{\Z}{\mathbb{Z}} 
\newcommand{\N}{\mathbb{N}} 
\newcommand{\T}{\mathbb{T}} 
\newcommand{\per}{\textup{per}}
\newcommand{\eff}{\textup{eff}}
\renewcommand{\Re}{\operatorname{Re}}
\renewcommand{\Im}{\operatorname{Im}}
\newcommand{\iu}{\mathrm{i}}
\newcommand{\eu}{\mathrm{e}}
\newcommand{\eps}{\varepsilon}
\newcommand{\vb}{\mathbf{v}}
\newcommand{\wb}{\mathbf{w}}
\newcommand{\ub}{\mathbf{u}}

\newcommand{\Non}{\mathcal{N}}

\newcommand{\bphi}{\boldsymbol{\phi}}
\newcommand{\bvphi}{\boldsymbol{\varphi}}
\newcommand{\bpsi}{\boldsymbol{\psi}}
\newcommand{\El}{\mathcal{L}}
\newcommand{\Be}{\mathcal{B}}
\newcommand{\de}{\mathrm{d}}
\newcommand{\ls}{\lesssim}

\DeclareMathOperator{\Ran}{Ran}

\allowdisplaybreaks[4]

\title{Dynamics of dissipative periodic optical waves in an external potential}

\author{Lukas Bengel$^{*}$}
\date{
    \small $^*$Institute for Analysis, Karlsruhe Institute of Technology, Englerstrasse 2, 76131 Karlsruhe, Germany; \texttt{lukas.bengel@kit.edu}
}

\begin{document}

\maketitle

\begin{abstract} 
We study the dynamics of periodic waves in a damped-driven nonlinear Schr\"odinger (NLS) equation perturbed by a small spatially dependent transport term. This equation arises in nonlinear optics as a model for light propagation in a passive optical cavity driven by a bichromatic laser source. Our main result shows that, for initial data close to a stable stationary wave of the translation-invariant unperturbed problem, the solution of the perturbed problem remains close to translated copies of this wave, with the translation parameter evolving according to an effective ordinary differential equation. If this effective equation admits a stable equilibrium, we prove that the solution of the perturbed NLS converges to a stationary state associated with that equilibrium. In particular, the asymptotically selected state attracts solutions with initial data that are not necessarily close to it, showing that its basin of attraction extends far beyond a small neighborhood of the state. The analysis is complicated by the loss of derivatives caused by the heterogeneous transport and combines the renormalization group method with a refined bootstrap argument that exploits parity-induced cancellations of quadratic nonlinear terms. Numerical simulations illustrate the analytical results.

\end{abstract}

\section{Introduction}
In this paper, we study the dynamical behavior of dissipative periodic optical waves in an external potential. The model considered is the damped-driven nonlinear Schrödinger (NLS) equation perturbed by a heterogeneous transport term,
\begin{align}\label{eq:LLE}
\begin{split}
    u_t &= \iu d u_{xx} + \eps V(x) u_x - (1 + \iu \omega) u + \iu |u|^2 u + f , \qquad (t,x) \in [0,\infty) \times \T\\
    u(0) &= u_0 \in H^1(\T),
\end{split}
\end{align}
where $\T = \R/(2\pi\Z)$ is the one-dimensional torus of length $2\pi$, $u \colon [0,\infty) \times \T \to \C$, $V \colon \T \to \R$ is a smooth periodic potential, and $\eps \geq 0$ is small. The parameters $d, \omega,f$ are real and satisfy $d \neq 0$ and $f>0$. If $d>0$ then the nonlinearity is focusing, whereas for $d<0$ it is defocusing. The NLS~\eqref{eq:LLE} with periodic potential has been derived in nonlinear optics~\cite{Bengel2024Pinning} to describe the light field inside a circular cavity which is excited by a bichromatic external laser source, see Section~\ref{sec:physics} for more details on the physical derivation. In this context,~\eqref{eq:LLE} is also known as the Lugiato-Lefever equation. Here, we address the following problem:

\emph{Let $\psi_0$ be a stable stationary solution of~\eqref{eq:LLE} for $\eps=0$, and let $u_0 \approx \psi_0$. Determine the dynamics of~\eqref{eq:LLE} for small $\eps >0$.}

The related problem of describing the motion of solitary waves under external potentials has been studied for purely dispersive systems, including the nonlinear Schr\"odinger equation~\cite{Froehlich2004Solitary} and the Korteweg-de Vries equation~\cite{Dejak2006Longtime}, using modulation theory. In these works, the solution is shown to remain close to a manifold of solitary waves over long but finite time intervals, while the evolution along this manifold is governed by an effective finite-dimensional system. For equations with dissipation that are second order in space, the renormalization group method was used in~\cite{Promislow2002Renormalization} to analyze the dynamics of localized solutions under perturbations that are bounded in $H^1(\R)$. It was shown that the solution remains for all times close to a translated localized profile whose position evolves according to a first-order ordinary differential equation.

Let us now recall some relevant results for~\eqref{eq:LLE}. If $\eps=0$, equation~\eqref{eq:LLE} is invariant under spatial translations,
$$
    u(t,x) \mapsto u(t,x+\sigma), \qquad x,\sigma \in \T, \ t \geq 0.
$$
This yields that every stationary solution $\psi_0$ for $\eps=0$ generates a one-dimensional family of solutions, which we denote by
$$
    \psi_\sigma(x) := \psi_0(x-\sigma), \qquad x,\sigma \in \T.
$$
In both dispersion regimes $d >0$ and $d<0$ the existence of stationary solutions for $\eps=0$ has been obtained in~\cite{Hakkaev2019Generation,Bengel2025Existence,Mandel2017Apriori,Delcey2018Instability,Delcey2018Periodic,Godey2017Bifurcation,Miyaji2011Stability} using bifurcation theory and spatial dynamics techniques. In the focusing case, the solutions constructed include periodic Turing rolls bifurcating from spatially constant states~\cite{Miyaji2011Stability,Delcey2018Periodic}, as well as far from equilibrium periodic $1$- and multi-pulses~\cite{Hakkaev2019Generation,Bengel2025Existence}. Spectral stability in this case has been obtained for small periodic solutions~\cite{Delcey2018Periodic,Miyaji2011Stability} and the 1- and multi-pulse solutions~\cite{Hakkaev2019Generation,Bengel2025Existence,Bengel2024Stability}. In the defocusing case, only existence results for small amplitude waves have been established rigorously while spectral stability has so far only been verified numerically~\cite{ParraRivas2016Dark}. Finally, nonlinear stability of spectrally stable solutions for both $d>0$ and $d<0$ has been proven in~\cite{Stanislavova2018Asymptotic,Haragus2024Nonlinear}.

When $\eps>0$, the external potential breaks the translation invariance and introduces a spatially dependent transport term. The persistence of stationary solutions of the unperturbed problem for small $\eps>0$ has been studied in~\cite{Bengel2024Pinning}. It was shown that stationary solutions bifurcate from the shifted solutions $\psi_{\sigma_*}$ provided that the shift $\sigma_*$ is a simple zero of the effective potential $\mathcal{V}_\eff$ (see also eq.~\eqref{eq:EffectivePotential} below) which is an averaged version of $V$ and depends on $\psi_0$. If in addition $\psi_0$ is spectrally stable and $\mathcal{V}_\eff'(\sigma_*)>0$, then the bifurcating solutions are shown to be spectrally stable for $\eps>0$ and also nonlinearly asymptotically stable. In particular, every solution of~\eqref{eq:LLE} with $\eps>0$ whose initial condition is sufficiently close to such a bifurcating state converges to this state as $t\to\infty$.

Next, we explain the main contribution of the present paper, which is a detailed description of the dynamics of periodic waves under the heterogeneous perturbation.

The first part of our main result can be stated informally as follows. If the initial datum satisfies 
$$
    \|u_0 - \psi_{\sigma_0}\|_{H^1} \leq C_0 \eps,
$$ 
for some $C_0>0, \sigma_0 \in \T$ and a stable stationary solution $\psi_0$ for $\eps = 0$, then the $H^1(\T)$-solution $u$ of~\eqref{eq:LLE} with $0< \eps \ll 1$ is global in time and obeys 
\begin{align}\label{eq:intro_dynamics}
    \|u(t) - \psi_0(\cdot-\sigma(t))\|_{H^1} \leq C_1 \eps, \qquad
    |\dot\sigma(t) + \eps \mathcal{V}_\text{eff}(\sigma(t))| < C_1 \eps^2, \qquad \text{for all } t \geq 0
\end{align}
with $C_1 > 0$ and $|\sigma(0) -  \sigma_0| \leq C_1 \eps$. This means that the solution stays close to the family of spatial translates $\{\psi_\sigma\}_{\sigma \in \T}$, while its motion along this family is, to leading order, governed by the reduced equation
\begin{align}\label{eq:intro_ODEVeff}
    \dot\sigma = - \eps \mathcal{V}_\eff(\sigma).
\end{align}
The result is illustrated in a subsequent simulation section, where we find excellent agreement between numerical computations and theory. Let us remark that if $\psi_0$ is strongly localized in space, its dynamics under~\eqref{eq:LLE} is expected to resemble that of a particle subject to the spatially dependent drift induced by the external potential~\cite[Chap.~7]{Bengel2025PhD}. This is consistent with the fact that $\mathcal{V}_\text{eff}$ approximates $V$ for pulse-like solutions, see~\cite[Rem.~14]{Bengel2024Pinning}. To establish our result, we extend the renormalization group method from~\cite{Promislow2002Renormalization} to~\eqref{eq:LLE}. The main challenge of the proof is that the perturbation term $\eps V(x) u_x$ is not bounded in $H^1(\T)$ and thus~\eqref{eq:LLE} does not fit into the framework considered in~\cite{Promislow2002Renormalization}. Indeed, if $\eps V(x)u_x$ is treated as an inhomogeneity in the mild formulation of~\eqref{eq:LLE}, one encounters a loss of one spatial derivative. This loss cannot be recovered from the linear evolution, since the linearized flow of~\eqref{eq:LLE} is not smoothing in $H^1(\T)$. To overcome this problem, we use the fact that the perturbation is linear in $u$ and prove uniform (in $\eps$ and $\sigma$) decay estimates for the linear problem
$$
    u_t = (\El + \eps V(x) \partial_x) u,
$$
where $\El$ denotes the linearization of~\eqref{eq:LLE} with $\eps = 0$ about the spectrally stable state $\psi_\sigma$. That is, instead of treating the perturbation as an inhomogeneity as in~\cite{Promislow2002Renormalization}, we include it into the linear part and analyze its influence there. Once the uniform linear decay estimates are established, we employ the renormalization group method to establish~\eqref{eq:intro_dynamics}. 

The second part of our main contribution concerns the asymptotic behavior of solutions. Let $\sigma_*$ be a simple zero of $\mathcal{V}_\eff$ with $\mathcal{V}_\eff'(\sigma_*)>0$. Then $\sigma_*$ is a stable equilibrium of~\eqref{eq:intro_ODEVeff} and it follows from~\cite{Bengel2024Pinning} that, for $\eps>0$ small, there exists a stable stationary solution $\psi_\text{stat}$ to~\eqref{eq:LLE} which is $O(\eps)$-close to $\psi_{\sigma_*}$. In view of~\eqref{eq:intro_dynamics}, these two observations suggest that solutions are first transported close to $\psi_\text{stat}$ according to~\eqref{eq:intro_dynamics} and then converge to this stationary state as $t\to\infty$ according to the local asymptotic stability result from~\cite{Bengel2024Pinning}. Establishing this conclusion, however, involves a significant difficulty. The local stability analysis in~\cite{Bengel2024Pinning} guarantees asymptotic convergence only for initial data lying in a $O(\eps)$-neighborhood measured in $H^1(\T)$ of the stationary state. By contrast, the estimates in~\eqref{eq:intro_dynamics} provide an $O(\eps)$ description of the solution $u$ with no control ensuring that the constant in this bound is sufficiently small. Consequently, these estimates alone do not imply that the solution $u$ eventually enters the local basin of attraction of $\psi_\text{stat}$. Nevertheless, under the additional assumption that $\psi_0$ is even, we overcome this problem and prove asymptotic convergence of $u$ to the stationary solution. The main idea is to improve an estimate of the form 
$$
    \|u(t_1) - \psi_\text{stat}\|_{H^1} \leq C \eps, \qquad t_1 \geq 0
$$ 
to 
$$
    \|u(t_2) - \psi_\text{stat}\|_{H^1} \leq C \eps^{3/2}, \qquad t_2>t_1
$$ 
over a time interval of length $t_2 - t_1 \approx - \log(\eps)/\eps$. The improved estimate ensures that $u(t_2)$ lies within the local basin of attraction of $\psi_\text{stat}$. To obtain this estimate, we exploit parity-induced cancellations of quadratic terms in the nonlinearity, which allow us to close the argument. Altogether, this shows that the basin of attraction of the asymptotically selected stationary state extends beyond a small neighborhood of that state, since the initial condition $u_0$ need not be close to it. Although the convergence result is established for~\eqref{eq:LLE}, the underlying technique should extend naturally to a broader class of dissipative systems subject to heterogeneous perturbations.

\subsection{Physical background of the problem}\label{sec:physics}
In nonlinear optics, the damped-driven NLS without an external potential has been derived as an amplitude equation to describe the light field in a circular optical cavity driven by a monochromatic external laser source~\cite{Haelterman1992Dissipative}. Such an experimental set-up can be used to generate optical signals known as frequency combs~\cite{Kippenberg2018Dissipative}, which have revolutionized frequency metrology and, more recently, enabled significant advances in optical communication~\cite{Marin-Palomo2017Microresonator} or optical neural networks~\cite{Xu2021TOPSPhotonic}. In this context, the normalized parameters have the following physical meaning: $d$ is the dispersion parameter, $f$ is the strength of the laser input, and $\omega = \omega_0- \omega_\text{p}$ is the frequency mismatch between the laser source frequency $\omega_\text{p}$ and the closest resonant frequency $\omega_{0}$ of the cavity. The linear damping term models cavity losses experienced by the light field.

The model with periodic potential~\eqref{eq:LLE} was derived in~\cite{Bengel2024Pinning} for a cavity driven by a bichromatic laser source consisting of a dominant frequency $\omega_\text{p}$ of strength $f$ and a second frequency component $\omega_{\tilde{\text{p}}}$ of strength $\tilde{f}$ which is much weaker in the sense that $|\tilde{f}| \ll |f|$. In the simplest setting, the bichromatic source generates a trapping potential of the form $V(x) = \omega_{\tilde{\text{p}}} - \omega_\text{p} + \omega_\text{FSR} - 2 d \tilde{f}\cos(x)/f$, where $\omega_\text{FSR}$ denotes the free spectral range determined by the cavity dispersion relation. The parameter $\eps$ controls the strength of the potential and is typically small. For details on the derivation of~\eqref{eq:LLE}, we refer to~\cite[Appendix A]{Bengel2024Pinning}. We note that~\eqref{eq:LLE} is formulated in the co-moving frame with speed $\omega_{\tilde{p}}- \omega_\text{p}$. 

Bichromatic pumping has been used in experiments to synchronize the repetition rate of frequency combs with an external laser source, thereby enabling chip-scale low-noise microwave generation~\cite{Kudelin2024Photonic}. To achieve synchronization, one first generates a stable optical signal using a monochromatic laser, corresponding to a stable stationary solution of~\eqref{eq:LLE} for $\eps = 0$. Once this signal has formed, the second, weaker laser is switched on. If the frequency of the second laser is suitably tuned such that $\omega_\text{p}-\omega_{\tilde{\text{p}}} \approx \omega_\text{FSR}$, the cavity field synchronizes with the resulting bichromatic source~\cite{Wildi2023Sideband}. Since~\eqref{eq:LLE} is formulated in the co-moving frame determined by the two laser frequencies, a synchronized signal corresponds to a stationary solution for $\eps>0$ and synchronization occurs if the cavity field $u$ converges towards such a stationary state as $t \to \infty$. In particular, our main result establishes synchronization for a large class of initial signals.

\subsection{Set-up, spectral assumption, and the main result}\label{sec:mainresult}
We denote by $C, C_1, \dots > 0$ absolute constants. For non-negative $\alpha, \beta$ we
use the notation $\alpha \ls \beta$ if $\alpha \leq  C\beta$ and write $\alpha \ll \beta$ to indicate that this constant is small. For a quantity $\gamma$ and non-negative $\delta$ we write $\gamma = O(\delta)$ if $\|\gamma\| \ls \delta$ in some norm that will be clear from the context. Moreover, we identify $\C$-valued functions $u$ with $\R^2$-valued functions $\ub = (\Re(u),\Im(u))^\top$ using boldface symbols. Spatial derivatives are sometimes denoted by $u'$ and time-derivatives by $\dot u$. For functions $u$ defined on $\T$ we sometimes write $u(\cdot -\sigma)$ with $\sigma \in \R$ where the argument is then understood modulo $2\pi$. Finally, with $B_\eps(x)$ we denote the open ball of radius $\eps$ around $x \in X$ in a normed space $X$, which will be clear from the context.

Introducing $\ub = (\Re(u),\Im(u))^\top$, we can write~\eqref{eq:LLE} as a real system
\begin{align}\label{eq:LLE_sys}
\begin{split}
    \ub_t &=  -d J\ub_{xx} + \eps V(x) \ub_x + N(\ub), \\
    \ub(0) &= \ub_0,
\end{split}
\end{align}
where $\ub_0 = (\Re(u_0),\Im(u_0))^\top$, $N(\ub) = J( \omega \ub - (\ub_1^2 + \ub_2^2) \ub) - \ub + f \mathbf{e}_1$, $\mathbf{e}_1 = (1,0)^\top$, and
\begin{align*}
    J = 
    \begin{pmatrix}
        0 & 1 \\ -1 & 0
    \end{pmatrix}.
\end{align*}
The unperturbed equation is obtained by setting $\eps = 0$ and reads
\begin{align}\label{eq:LLE0_sys}
    \ub_t = -d J\ub_{xx} + N(\ub).
\end{align}
Throughout the proofs of this paper, we work with the formulation~\eqref{eq:LLE_sys}, which has the advantage, that the linearization about a stationary state is a $\C$-linear operator. For $\ub\in L^\infty(\T)$ we define
\begin{align}
\begin{split}
    \El_\eps(\ub) \colon H^2(\T) \subset L^2(\T) \to L^2(\T) , \qquad
    \El_\eps(\ub) = -d J \partial_x^2 + \eps V(x) \partial_x + \partial_\ub N(\ub).
\end{split}
\end{align}
If $\ub$ is a stationary solution of~\eqref{eq:LLE_sys}, then the linearization about $\ub$ is given by $\El_\eps(\ub)$.
Since the embedding $H^2(\T) \hookrightarrow L^2(\T)$ is compact, $\El_\eps(\ub)$ is a Fredholm operator of index zero and has compact resolvents.
Thus, its spectrum consists entirely of isolated eigenvalues of finite algebraic multiplicity.

Next, we state the spectral assumption on the stationary state. Let $\psi_0 \in H^2(\T)$ be a stationary solution of~\eqref{eq:LLE} for $\eps = 0$ and write $\bpsi_0 = (\Re(\psi_0),\Im(\psi_0))^\top$. 

\begin{Assumption}\label{Assum1}
For $\bpsi_0 = (\Re(\psi_0),\Im(\psi_0))^\top \in H^2(\T)$ we require the following properties.
    \begin{itemize}
        \item[(i)] Spectral stability: There exists $\rho>0$ such that 
        $$
            \sigma(\El_0(\bpsi_0)) \subset \{\lambda \in \C : \Re(\lambda) \leq -\rho\} \cup\{0\};
        $$
        \item[(ii)] Nondegeneracy: $0\in \sigma(\El_0(\bpsi_0))$ is an algebraically simple eigenvalue and we denote by $\bphi_0^* \in H^2(\T)$ the unique solution of $\El_0(\bpsi_0)^*\bphi_0^* = 0$ which satisfies $\langle \bpsi_0',\bphi_0^* \rangle_{L^2} = 1$.
    \end{itemize}
\end{Assumption}
Examples of stationary solutions for $d>0$ satisfying Assumption~\ref{Assum1} can be found in~\cite{Delcey2018Periodic,Hakkaev2019Generation,Bengel2025Existence}.

Let us fix a stationary solution $\psi_0 \in H^2(\T)$ of~\eqref{eq:LLE} for $\eps = 0$ which satisfies Assumption~\ref{Assum1} and denote the family of shifts by $\{\bpsi_\sigma\}_{\sigma \in \T}$ and let $\bphi_\sigma^*:= \bphi_0^*(\cdot-\sigma)$. Then, $\El_0(\bpsi_\sigma)^*\bphi_\sigma^* = 0$ as well as $\langle \bpsi_\sigma',\bphi_\sigma^* \rangle_{L^2} =1$ for all $\sigma \in \R$. We now define the effective potential $\mathcal{V}_\eff \colon \R \to \R$ by setting
\begin{align}\label{eq:EffectivePotential}
    \mathcal{V}_\eff(\sigma) 
    := \int_\T V(y) \ \langle\bpsi_\sigma'(y), \bphi_\sigma^*(y)\rangle_{\R^2} \ \de y.
\end{align}
Note that $\mathcal{V}_\eff$ is $2\pi$-periodic and smooth. We recall from~\cite{Bengel2024Pinning} existence and stability results for stationary solution for $\eps>0$ that bifurcate from $\{\bpsi_\sigma\}_{\sigma\in \T}$ and demand the existence of a simple zero of $\mathcal{V}_\eff$.

\begin{Corollary}[\!{\!\cite[Thms.~3,8,9]{Bengel2024Pinning}}]\label{cor:ExistenceStationary}
    Let $\psi_0 \in H^2(\T)$ be a stationary solution of~\eqref{eq:LLE} for $\eps = 0$ that satisfies Assumption~\textup{\ref{Assum1}}. Assume that the effective potential $\mathcal{V}_\eff$ admits a simple zero $\sigma_*$. 
    Then there exists $\eps_0,C>0$ such that for all $\eps \in (0,\eps_0)$ the following holds.
    \begin{itemize}
        \item[(i)] There exists a stationary solution $\psi_{\sigma_*,\eps} \in H^2(\T)$ of~\eqref{eq:LLE} that satisfies the estimate
        \begin{align*}
            \|\psi_{\sigma_*,\eps} - \psi_{\sigma_*}\|_{H^2} \leq C \eps.
        \end{align*}
        Moreover, $\psi_{\sigma_*,\eps}$ is strongly spectrally stable if $\mathcal{V}_\eff'(\sigma_*) > 0$, that is
        $$
            \sigma(\El_\eps(\bpsi_{\sigma_*,\eps}))  \subset \{\lambda \in \C : \Re(\lambda) < -\eta\}
        $$
        for some $\eta>0$ and $\bpsi_{\sigma_*,\eps} = (\Re(\psi_{\sigma_*,\eps}),\Im(\psi_{\sigma_*,\eps}))$. Conversely, $\psi_{\sigma_*,\eps}$ is spectrally unstable if $\mathcal{V}_\eff'(\sigma_*) < 0$.
        \item[(ii)] If $\psi_{\sigma_*,\eps}$ is spectrally stable, then it is also locally asymptotically stable in $H^1(\T)$. 
        That is, there exist constants $\eta,C,\delta>0$ such that for all $u_0 \in H^1(\T)$ with 
        $$
            \|u_0 - \psi_{\sigma_*,\eps}\|_{H^1} < \delta
        $$ 
        there exists a unique global $H^1(\T)$-solution $u \in C([0,\infty),H^1(\T))$ of~\eqref{eq:LLE} that obeys the estimate
        \begin{align*}
            \|u(t) - \psi_{\sigma_*,\eps}\|_{H^1} \leq C \eu^{-\eta t} \|u_0 - \psi_{\sigma_*,\eps}\|_{H^1}  \qquad \text{for all } t \geq 0.
        \end{align*}
    \end{itemize}
\end{Corollary}

We are now ready to formulate the main result of the present paper. 

\begin{Theorem}\label{thm:main}
Let $\psi_0\in H^2(\T)$ be a stationary solution to~\eqref{eq:LLE} for $\eps=0$ which satisfies Assumption~\textup{\ref{Assum1}}. Then there exist constants $\eps_0,C_0,C_1,\kappa,K>0$ such that for any $\eps \in (0,\eps_0)$ the following holds. Let $v_0 \in H^1(\T)$ with
$$
    \|v_0\|_{H^1} \leq C_0 \eps
$$
and let $\sigma_0 \in \R$. 
\begin{itemize}
    \item[(i)] \textbf{Modulational dynamics:} There exists a unique global $H^1(\T)$-solution $u\in C([0,\infty), H^1(\T))$ to~\eqref{eq:LLE} with initial condition 
    $$
        u_0 = \psi_{\sigma_0} + v_0
    $$
    and it admits the decomposition
    \begin{align*}
        u(t) = \psi_{\sigma(t)} + v(t), \qquad
        \|v(t)\|_{H^1} \leq C_1 \eps \qquad \text{for all } t \geq 0.
    \end{align*}
    Here, $\sigma \in C^1([0,\infty),\R)$ and satisfies 
    \begin{align}\label{eq:ODEsigma}
        |\dot\sigma(t) + \eps \mathcal{V}_\eff(\sigma(t))| \leq C_1 \eps^2 \qquad \text{for all } t \geq 0
    \end{align}
    and $|\sigma(0) - \sigma_0| \leq C_1\eps$.
    \item[(ii)]\textbf{Asymptotic stability:} Let $\sigma_* \in \R$ be a simple zero of $\mathcal{V}_\eff$ with $\mathcal{V}_\textup{eff}(\sigma_*) = 0$ and $\mathcal{V}_\textup{eff}'(\sigma_*)> 0$. We define 
    $$
        \mathcal{I}_{\sigma_*} := \bigcup\Big\{(\sigma_*-\delta_-,\sigma_*+\delta_+) :  
        \begin{array}{c}
            \delta_-,\delta_+ > 0,\ \mathcal{V}_\eff^{-1} ([-2C_1\eps,2C_1\eps]) \cap (\sigma_*-\delta_-,\sigma_*+\delta_+) \\
            \text{ is connected for all } \eps < \eps_0      
        \end{array}
        \Big\}.
    $$
    Let further $\psi_0$ be an even function and $\sigma(0) \in \mathcal{I}_{\sigma_*}$. We denote by $\psi_{\sigma_*,\eps}$ the stationary solution of~\eqref{eq:LLE} bifurcating from $\psi_{\sigma_*}$ given in Corollary~\textup{\ref{cor:ExistenceStationary}}. Then, we obtain the asymptotic convergence 
    \begin{align*}
        \|u(t) - \psi_{\sigma_*,\eps}\|_{H^1} \leq C_1\frac{\eu^{K\eps^{-1}} }{\eps}  \eu^{-\eps\kappa t} \qquad \text{for all } t \geq 0.
    \end{align*}
\end{itemize}
\end{Theorem}

\begin{Remark}
If all assumptions of Theorem~\ref{thm:main} are satisfied, we can combine the statements (i) and (ii) to obtain a global description of the solution. Then for all $t \ls \eps^{-2} - 2 \log(\eps)/(\eps\kappa)$, part~(i) yields that the solution evolves according to~\eqref{eq:ODEsigma} which, since $\sigma(0) \in \mathcal{I}_{\sigma_*}$ holds, transports the periodic wave to the neighborhood of $\psi_{\sigma_*,\eps}$. For $ \eps^{-2} - 2 \log(\eps)/(\eps\kappa) \ls  t$ the estimate in (ii) is sharper than the solution bound in (i) and yields that $u$ converges asymptotically fast towards $\psi_{\sigma_*,\eps}$ at a slow exponential rate.
\end{Remark}

\begin{Remark}
If $\mathcal{V}_\eff$ has no zero, then the ODE $\dot\sigma = -\eps \mathcal{V}_\eff(\sigma)$ has no equilibrium. Thus, $\sigma$ is strictly monotone and since we are solving~\eqref{eq:LLE} on the torus, the solution $u$ is to leading order time-periodic.
In particular, $u$ does not converge to a stationary solution of~\eqref{eq:LLE} as $t \to \infty$.
\end{Remark}

\begin{Remark}
Let $V(x)=a+b\cos(x)$ with $a,b\in\R$ and $b\neq0$. By~\cite[Rem.~13]{Bengel2024Pinning}, the effective potential is generically of the form $\mathcal{V}_\eff(\sigma)=a+\beta\cos(\sigma)$ with $\beta \neq 0$. We compute $\mathcal{I}_{\sigma_*}$ for this specific example. Assume for simplicity that $\beta>0$, let $|a|<\beta$ and set $\theta:=\arccos(-a/\beta)$. Then $\sigma_*=2\pi-\theta$ is the unique zero of $\mathcal{V}_\eff$ on $\T$ satisfying $\mathcal{V}_\eff'(\sigma_*)>0$. Set $\eta_0:=2C_1\eps_0>0$ and assume that $\eta_0<\beta-|a|$. The subset of $\mathcal{V}_\eff^{-1}([-\eta_0,\eta_0])$ containing the unstable zero $\theta$ is given by $[\theta_-,\theta_+]$, where $\theta_-=\arccos((\eta_0-a)/\beta)$ and $\theta_+=\arccos(-(\eta_0+a)/\beta)$. Consequently, $ \mathcal{I}_{\sigma_*}= \T \setminus [\theta_-,\theta_+]_{\sim}$, where $\sim$ indicates the identification modulo $2\pi$. In particular, $\mathcal{I}_{\sigma_*}$ is not a small $O(\eps)$-neighborhood of $\sigma_*$.
\end{Remark}

\begin{Remark}
The parity assumption in (ii) imposed on $\psi_0$ is used in the proof of Lemma~\ref{lem:estimates2} and helps us to improve an estimate due to the cancellation of $\int_\T \partial_\ub^2 N(\bpsi_0) [\bpsi_0']^2 \overline{\bphi_0^*} \de x =0$. 
\end{Remark}

\subsection{Organization of paper}
The remainder of this paper is organized as follows. Section~\ref{sec:numeric} contains numerical simulations illustrating our main result. In Section~\ref{sec:linear} we derive a spectral decomposition of the linearized operators $\El_\eps(\bpsi_\sigma)$ and $\El_\eps(\bpsi_{\sigma_*,\eps})$ and prove uniform linear decay estimates for the corresponding semigroups. In Section~\ref{sec:ModulationDynamics} we present the proof of Theorem~\ref{thm:main}. Part (i) is established in Section~\ref{sec:ModulationDynamics}. The proof of part (ii) is presented in Section~\ref{sec:AsmyptoticConvergence}, which is again divided into an analysis in the intermediate time regime (Section~\ref{sec:IntermediatRegime}) and in the asymptotic time regime (Section~\ref{sec:AsymptoticRegime}). Finally, Appendix~\ref{sec:Append_IFT} contains a version of the implicit function theorem with uniform constants which is needed for the proofs. 

\section{Numerical simulations}\label{sec:numeric}
In this section, we present numerical simulations to illustrate the two dynamical scenarios described by Theorem~\ref{thm:main}: persistent spatially dependent drift when the effective equation has no equilibrium and trapping when it admits a stable equilibrium. We consider both the focusing and defocusing regimes. We first computed stable bright (for $d>0$) and dark (for $d<0$) stationary pulse solutions to~\eqref{eq:LLE} with $\eps = 0$ using numerical bifurcation methods. In the focusing case, we obtained a stable pulse for $d= 0.1, \omega = 4.8$ and $f=2$, while in the defocusing case, we obtained a stable pulse for $d=-0.1, \omega = 3.5$ and $f =2$. We imposed symmetry constraints to ensure that both solutions are even in $x$ and verified that they satisfy Assumption~\ref{Assum1}. We then used these stationary pulses as initial data for the time integration of~\eqref{eq:LLE}. In each simulation, we compare the observed pulse position with the solution of the reduced equation $\dot \sigma = - \eps \mathcal{V}_\eff(\sigma)$.

\begin{figure}[!b]
    \centering
    \includegraphics[width=\textwidth]{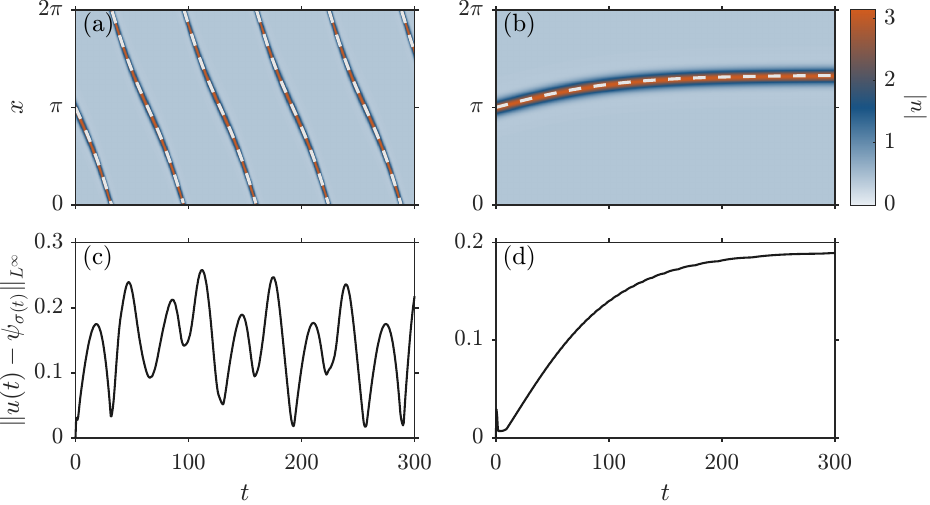}
    \caption{Panels (a) and (b) show the space-time evolution $|u(x,t)|$ starting from a stable pulse in the focusing case for two choices of the external potential: $\eps V(x) = 0.1+0.02\cos(x)$ in the left column, and $\eps V(x) = 0.01+ 0.02\cos(x)$ in the right column. The dashed line indicates the position computed from the ODE $\dot\sigma(t) = - \eps \mathcal{V}_\eff(\sigma(t))$. Panels (c) and (d) show the $L^\infty$-error $\|u(t)-\psi_{\sigma(t)}\|_{L^\infty}$ which is controlled by Theorem~\ref{thm:main}. Parameters used in this simulation are $d=0.1, \omega = 4.8, f=2$.} 
    \label{fig1}
\end{figure}

Figure~\ref{fig1} presents the results for the bright pulse in the focusing case. In the left column, we set $\eps V(x) = 0.1+0.02\cos(x)$. For this choice, the numerically computed effective potential has no zeros. Panel (a) shows the space-time evolution of the solution, which exhibits a nonuniform spatially dependent drift around the torus over the full simulation time. The dashed line indicates the solution $\sigma$ of the reduced ODE and closely follows the numerically observed pulse position. Panel~(c) shows that the error $\|u(t) - \psi_{\sigma(t)}\|_{L^\infty}$ remains small and oscillatory. In the right column, we set $\eps V(x) = 0.01+0.02\cos(x)$. Here, the numerically computed effective potential admits a stable zero $\sigma_*$ and hence the reduced equation predicts transport towards $\sigma_*$. Panel~(b) confirms this prediction as the pulse initially drifts and subsequently becomes trapped at $\sigma_*$. Once again, the reduced trajectory closely agrees with the observed pulse position. As the pulse relaxes towards the stationary state, the error visible in panel (d) approaches the plateau $\|\psi_{\sigma_*}- \psi_{\sigma_*,\eps}\|_{L^\infty}$.

\begin{figure}[!t]
    \centering
    \includegraphics[width=\textwidth]{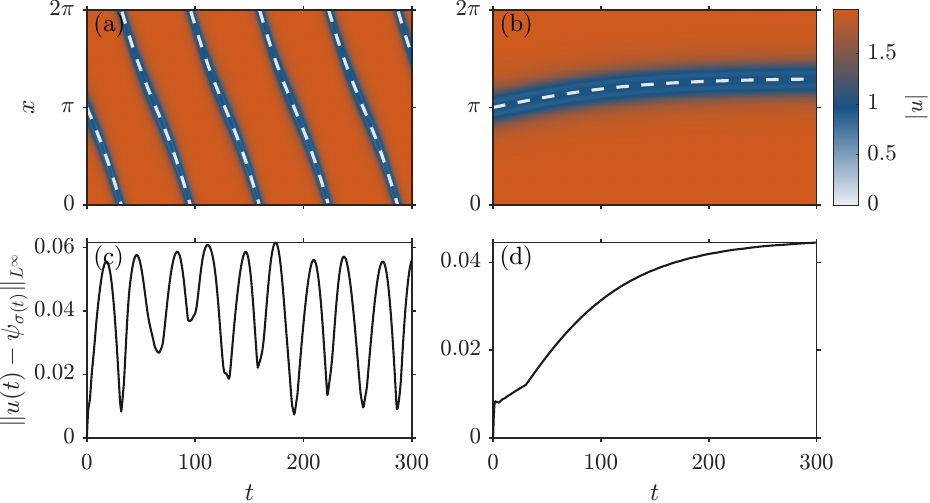}
    \caption{Same as in Figure~\ref{fig1} but for the defocusing case. Parameters are $d=-0.1, \omega = 3.5, f=2$ and $\eps V(x) = 0.1+0.02\cos(x)$ (left column), $\eps V(x) = 0.01+ 0.02\cos(x)$ (right column).} 
    \label{fig2}
\end{figure}

Figure~\ref{fig2} presents the simulation results for the defocusing case and with the dark pulse as initial data. As before, only the potential $V$ is varied with the choice $\eps V(x) = 0.1 + 0.02 \cos(x)$ for the left column and $\eps V(x) = 0.01 + 0.02 \cos(x)$ for the right column. Panels~(a) and~(c) show the oscillatory behavior of the dark pulse together with the approximation error in the case where the effective potential is sign-definite. Again, the position is accurately predicted by the effective ODE for $\sigma$. Panels~(b) and~(d) show that, when the effective potential possesses a stable zero, the dark pulse becomes trapped and approaches a stationary solution as $t \to \infty$.

\section{Spectral and linear analysis}\label{sec:linear}
Throughout this section, $\bpsi_0 \in H^2(\T)$ satisfies Assumption~\ref{Assum1}. In order to prove Theorem~\ref{thm:main}, we first need to establish some preliminary results on the spectrum of $\El_\eps(\bpsi_\sigma)$ as well as on linear estimates for the semigroup that is generated by $\El_\eps(\bpsi_\sigma)$. In addition to that, the proof of part (ii) of Theorem~\ref{thm:main} uses spectral information and linear decay estimates for the stationary solution of Corollary~\ref{cor:ExistenceStationary}. These results were established (in a slightly different form than needed here) in~\cite{Bengel2024Pinning}, and are therefore stated again in this section.

\subsection{Spectral analysis}
We start with the following perturbation result for the simple zero eigenvalue of $\El_0(\bpsi_\sigma)$.
\begin{Lemma}\label{lem:perturbation_eigenvalue}
    There exist $\eps_0,\delta_0>0$ such that for all $\sigma,\sigma_1,\sigma_2 \in \R$, $\eps \in [0,\eps_0)$ the following holds.
    \begin{itemize}
        \item[(a)] There exists $\bphi_{\sigma,\eps},\bphi_{\sigma,\eps}^* \in H^2(\T)$ and $\lambda_{\sigma,\eps} \in \R$ such that
        \begin{enumerate}
            \item $\sigma(\El_{\eps}(\bpsi_\sigma)) \cap B_{\delta_0}(0) = \sigma(\El_{\eps}(\bpsi_\sigma)^*) \cap B_{\delta_0}(0) = \{\lambda_{\sigma,\eps}\}$;
            \item $\El_{\eps}(\bpsi_\sigma) \bphi_{\sigma,\eps} = \lambda_{\sigma,\eps}\bphi_{\sigma,\eps}$ and $\El_{\eps}(\bpsi_\sigma)^* \bphi_{\sigma,\eps}^* = \lambda_{\sigma,\eps}\bphi_{\sigma,\eps}^*$;
            \item $|\lambda_{\sigma,\eps}|, \|\bpsi_\sigma' - \bphi_{\sigma,\eps}\|_{H^1}, \|\bphi_{\sigma}^* - \bphi_{\sigma,\eps}^*\|_{H^1} \ls \eps$;
            \item $\|\bpsi_{\sigma_1}' - \bpsi_{\sigma_2}'\|_{H^2}, \|\bphi_{\sigma_1}^* - \bphi_{\sigma_2}^*\|_{H^2} \ls |\sigma_1-\sigma_2|$;
            \item $\langle \bphi_{\sigma,\eps},\bphi_{\sigma,\eps}^* \rangle_{L^2}= 1$.
        \end{enumerate}
        \item[(b)] Assume in addition that $\sigma_*$ is a simple zero of $\mathcal{V}_\eff$ and denote by $\bpsi_{\sigma_*,\eps} \in H^2(\T)$ the associated stationary solution obtained in Corollary~\ref{cor:ExistenceStationary}.
        Then there exist $\bvphi_{\eps},\bvphi_{\eps}^* \in H^2(\T)$ and $\lambda_{\eps} = - \lambda_1 \eps + O(\eps^2) \in \R$, $\lambda_1 = \mathcal{V}_\eff'(\sigma_*)$ such that
        \begin{enumerate}
            \item $\sigma(\El_{\eps}(\bpsi_{\sigma_*,\eps})) \cap B_{\delta_0}(0) = \sigma(\El_{\eps}(\bpsi_{\sigma_*,\eps})^*) \cap B_{\delta_0}(0) = \{\lambda_{\eps}\}$;
            \item $\El_{\eps}(\bpsi_{\sigma_*,\eps}) \bvphi_{\eps} = \lambda_{\eps}\bvphi_{\eps}$ and $\El_{\eps}(\bpsi_{\sigma_*,\eps})^* \bvphi_{\eps}^* = \lambda_{\eps}\bvphi_{\eps}^*$;
            \item $ \|\bpsi_{\sigma_*}' - \bvphi_{\eps}\|_{H^2}, \|\bphi_{\sigma_*}^* - \bvphi_{\eps}^*\|_{H^2} \ls \eps$;
            \item $\langle \bvphi_{\eps},\bvphi_{\eps}^* \rangle_{L^2}= 1$.
        \end{enumerate}
    \end{itemize}
    Moreover, $\bphi_{\sigma,\eps},\bphi_{\sigma,\eps}^*,\bvphi_\eps,\bvphi_\eps^*$ depend smoothly on $\eps$.
\end{Lemma}
\begin{proof}
    We start with the proof of part (a). The claims 1.-3.~for the operator $\El_\eps(\bpsi_\sigma)$ follow as in~\cite[Prop.~1.7.2]{Kielhoefer2012Introduction}. 
    However, to obtain uniform bounds in $\sigma$, we need to apply a version of the Implicit Function Theorem~\ref{thm:IFT} with uniform bounds, which is stated in Appendix~\ref{sec:Append_IFT}, to the function
    $$
        G \colon \Ran(P) \times \R \times \R \to L^2(\T), \quad
        G(\vb,\lambda,\eps):= \El_{\eps}(\bpsi_\sigma) (\bpsi_\sigma' + \vb) - \lambda (\bpsi_\sigma' + \vb),
    $$
    where the projection $P \colon L^2(\T) \to L^2(\T)$ is given by $P\ub = \ub - \langle \ub, \bphi_\sigma^* \rangle_{L^2} \bpsi_\sigma'$. Checking the hypothesis of Theorem~\ref{thm:IFT} is straightforward and thus we omit the details here. For the adjoint operator $\El_\eps(\bpsi_\sigma)^*$, we proceed similarly. Moreover, 4 is an application of the mean value theorem, where we note that $\psi_0$ is in fact $H^4$-smooth and $\bpsi_0^*\in H^3(\T)$. Assertion 5 is a normalization, which is possible since $\lambda_{\eps}$ is a simple eigenvalue.

    For part (b) we refer to~\cite[Sect.~5]{Bengel2024Pinning}.
\end{proof}

We define the spectral projections
\begin{align}\label{def:spec_proj}
\begin{split}
    &\Pi_{\sigma,\eps}^\textup{c}
    \colon L^2(\T) \to L^2(\T), \qquad
    \Pi_{\sigma,\eps}^\textup{c}\ub
    := \langle \ub,\bphi_{\sigma,\eps}^* \rangle_{L^2}\,\bphi_{\sigma,\eps}, \qquad
    \Pi_{\sigma,\eps}^\textup{s} := I - \Pi_{\sigma,\eps}^\textup{c},
    \\
    &\Pi_{\eps}^\textup{c}
    \colon L^2(\T) \to L^2(\T), \qquad
    \Pi_{\eps}^\textup{c}\ub
    := \langle \ub,\bvphi_{\eps}^* \rangle_{L^2}\,\bvphi_{\eps}, \qquad
    \Pi_{\eps}^\textup{s} := I - \Pi_{\eps}^\textup{c},
\end{split}
\end{align}
onto the central and stable spaces $\Ran(\Pi_{\sigma,\eps}^\textup{c})=:X_{\sigma,\eps}^\textup{c}$, $\Ran(\Pi_{\eps}^\textup{c})=:X_{\eps}^\textup{c}$ and $\Ran(\Pi_{\sigma,\eps}^\textup{s})=:X_{\sigma,\eps}^\textup{s}$, $\Ran(\Pi_{\eps}^\textup{s})=:X_{\eps}^\textup{s}$, respectively. Using the spectral projections, we will decompose the $C^0$-semigroups of the corresponding operators, which will be useful in the upcoming nonlinear analysis.

\subsection{Linear estimates}
In this section we derive exponential decay estimates for the linearized equations. 
To this end, let $\ub \in H^2(\T)$. 
It follows from~\cite[Sect.~5.9]{Bengel2024Pinning} that the operator $\El_\eps(\ub)$ generates a $C^0$-semigroup on $H_\per^s(0,2\pi)$ for $s=0,1,2$.
We denote the semigroup on $L^2(\T)$ by $\{\eu^{\El_\eps(\ub) t} \}_{t \geq 0}$ and note that the semigroup on $H_\per^s(0,2\pi)$ for $s=1,2$ can be obtained by restriction to the respective invariant spaces.
In case where the assumptions of Corollary~\ref{cor:ExistenceStationary} are satisfied, let $\bpsi_{\sigma_*,\eps}$ be a spectrally stable solution for $\eps$ sufficiently small.
For $\sigma \in \R$ we define the $C^0$-semigroups 
$$
    S_{\sigma,\eps}(t):= \eu^{\El_{\eps}(\bpsi_\sigma)t}\Pi_{\sigma,\eps}^\textup{s}
    \qquad\text{and}\qquad 
    S_\eps(t) = \eu^{\El_\eps(\bpsi_{\sigma_*,\eps}) t}.
$$
Then, the following exponential decay estimates hold.
\begin{Lemma}\label{lem:LinearDecay}
    There exists $\kappa \in (0,\rho)$, $\eps_0>0$ such that for all $\eps \in (0,\eps_0)$ and all $\sigma \in \R$ the linear estimates
    \begin{enumerate}
        \item $\|S_{\sigma,\eps}(t)\|_{H^1 \to H^1} \ls  \eu^{-\kappa t}$
        \item $\|S_{\eps}(t)\Pi_\eps^\textup{s}\|_{H^1 \to H^1} \ls  \eu^{-\kappa t}$
        \item $\|S_{\eps}(t)\|_{H^1 \to H^1} \ls  \eu^{-\eps\kappa t}$
    \end{enumerate}
    hold for all $t \geq 0$.
\end{Lemma}
\begin{proof}
    We start with the proof of the first statement and note that the second statement can be proven by the same arguments.
    The idea is to use the stability theorem of Gearhart-Pr\"uss-Greiner~\cite[Sect.~5 Thm.~1.1]{Engel2000One}. 
    However, we need to track the $\eps$ and $\sigma$-dependencies of the constants in the estimates.

    It follows from~\cite[Lem.~19]{Bengel2024Pinning} that there exist constants $\varpi,C>0$ ($\eps,\sigma$ independent) such that $\|S_{\sigma,\eps}(t)\|_{L^2\to L^2} \leq C \eu^{\varpi t}$ and
    $$
        \|(\lambda-\El_\eps(\bpsi_\sigma))^{-1} \Pi_{\sigma,\eps}^\textup{s}\|_{L^2 \to L^2} \leq C, \qquad  \text{ for all } \lambda\in \C \text{ with }\Re(\lambda)\geq -\frac{\rho}{2}
    $$
    provided $\eps \ll 1$.
    Thus, following word by word the proof in~\cite[Sect.~5 Thm.~1.1]{Engel2000One} we find $C'>0$ independent of $\eps,\sigma$ such that $\|S_{\sigma,\eps}(t)\|_{L^2 \to L^2} \leq C' t^{-1}$. 
    Using~\cite[Sect.~2 Prop.~2.2, Lem.~2.3]{Engel2000One} we then obtain $\|S_{\sigma,\eps}(t)\|_{L^2 \to L^2} \leq C \eu^{-\kappa t}$ for $\kappa,C>0$ independent of $\sigma$ and $\eps$. 
    To upgrade the $L^2\to L^2$-estimate to an $H^1\to H^1$-estimate, we exploit complex interpolation as in~\cite[Lem.~22]{Bengel2024Pinning}.
    This yields the first estimate upon noting that the constants remain $\eps$ and $\sigma$-independent after interpolating.

    To prove the third statement, we decompose the $C^0$-semigroup and apply the estimate (ii) to obtain
    \begin{align*}
        \|S_\eps(t)\|_{H^1 \to H^1} 
        \leq \|S_\eps(t)\Pi_\eps^\textup{s}\|_{H^1 \to H^1} 
        + \|S_\eps(t)\Pi_\eps^\textup{c}\|_{H^1 \to H^1}
        \ls \eu^{-\kappa t} + \eu^{\lambda_\eps t}\| \langle \cdot,\bvphi_\eps^* \rangle_{L^2} \bvphi_\eps\|_{H^1 \to H^1} 
        \ls \eu^{-\eps \kappa t},
    \end{align*}
    upon adjusting the value of $\kappa>0$ if necessary.
    Here we note that the last estimate follows by combining the expansion from Lemma~\ref{lem:perturbation_eigenvalue} (b), $\lambda_\eps = - \lambda_1 \eps + O(\eps^2)$ with the spectral stability assumption $\lambda_1 = \mathcal{V}_\eff'(\sigma_*) >0$.
\end{proof}

\section{Nonlinear dynamics: proof of Theorem~\ref{thm:main}}\label{sec:NonlinearDynamics}
In this section, we give the proof of Theorem~\ref{thm:main}. The section is structured as follows. We start with a lemma that allows us to decompose initial data into a stable part and a central part according to the decompositions established in Section~\ref{sec:linear}. Then, in Section~\ref{sec:ModulationDynamics} we employ the renormalization group method to prove part (i) of Theorem~\ref{thm:main}. In Section~\ref{sec:AsmyptoticConvergence}, we prove the asymptotic stability result of Theorem~\ref{thm:main} by establishing refined nonlinear estimates and improving errors on an intermediate time regime.

Let us assume that $\bpsi_0 \in H^2(\T)$ satisfies Assumption~\ref{Assum1} and let $\bpsi_{\sigma_*,\eps}$ be given as in Theorem~\ref{thm:main}~(ii).
Throughout the proof, we repeatedly employ a change of coordinates in a neighborhood of the manifold $\{\bpsi_\sigma\}_{\sigma \in \R}$. 
This will allow us to decompose the solution of~\eqref{eq:LLE_sys} into a part that can be controlled through the use of linear estimates of Lemma~\ref{lem:LinearDecay}
and a part that evolves on the manifold $\{\bpsi_\sigma\}_{\sigma \in \R}$. 
The following lemma establishes this decomposition.
\begin{Lemma}\label{lem:inital_decomposition}
    There exist $\eta_0,\eps_0>0$ such that for all $\eps \in [0,\eps_0)$ and all $v \in H^1(\T)$ with $\|v\|_{H^1} \leq 1$ the following holds.
    \begin{enumerate}
        \item For all $\sigma_0\in \R$ and $\eta \in [0,\eta_0)$, there exists a unique $\sigma\in \R$ with $|\sigma- \sigma_0| \ls \eta$ and
        \begin{align*}
            \bpsi_{\sigma_0} + \eta \vb - \bpsi_{\sigma} \in X_{\sigma,\eps}^\textup{s}.
        \end{align*}
        Furthermore, if there is $\sigma_1 \in \R$ such that $\vb \in X_{\sigma_1,\eps}^\textup{s}$, then
        $
            |\sigma-\sigma_0|  \ls  \eta (|\sigma_1-\sigma_0| + \eps).
        $
        \item If $|\sigma_0-\sigma_*| \ls \eps_0$, then for all $\eta \in [0,\eta_0)$ there exists a unique $\varsigma \in \R$ with $|\varsigma- \sigma_0|\ls \eta $ and
        \begin{align*}
            \bpsi_{\sigma_0} + \eta \vb - \bpsi_{\varsigma} \in X_{\eps}^\textup{s}.
        \end{align*}
    \end{enumerate}
\end{Lemma}
\begin{proof}
    (i): Let $\eps_0>0$ be as in Lemma~\ref{lem:perturbation_eigenvalue}. 
    For $\eps \in [0, \eps_0)$ we define the map $G \colon \R^2 \to \R$, $G(\sigma,\eta) := \langle \bpsi_{\sigma_0} + \eta \vb - \bpsi_{\sigma}, \bphi_{\sigma,\eps}^* \rangle_{L^2}$. 
    Note that $\bpsi_{\sigma_0} + \eta \vb - \bpsi_{\sigma} \in X_{\sigma,\eps}^\textup{s}$ if and only if $G(\sigma,\eta) = 0$. 
    On the one hand, we have $G(\sigma_0,0) = 0$ and on the other hand it holds 
    $\partial_\sigma G(\sigma_0,0) = \langle \bpsi_{\sigma_0}',\bphi_{\sigma_0,\eps}^* \rangle_{L^2} = 1 + O(\eps) \not = 0$ 
    upon choosing $\eps_0>0$ smaller if necessary. 
    By Theorem~\ref{thm:IFT}, there exists $\eta_0>0$ depending only on $\eps_0$ and a smooth branch 
    $[0,\eta_0) \ni \eta\mapsto \sigma(\eta) \in \R$ with $\sigma(0) = \sigma_0$ such that $G(\sigma(\eta),\eta) = 0$ for all $\eta \in [0, \eta_0)$. 
    If further $\vb \in X_{\sigma_1,\eps}^\textup{s}$, we have using H\"older's inequality and Lemma~\ref{lem:perturbation_eigenvalue},
    $$
        |\partial_\eta G(\sigma,\eta)| = |\langle \vb, \bphi_{\sigma,\eps}^* \rangle_{L^2}| \leq |\langle \vb, \bphi_{\sigma_1,\eps}^*-\bphi_{\sigma,\eps}^* \rangle_{L^2}|
        \leq \|\bphi_{\sigma_1,\eps}^*-\bphi_{\sigma,\eps}^* \|_{L^2} \ls |\sigma_1-\sigma| + \eps.
    $$
    Thus, using $\partial_\eta \sigma(\eta) = - \partial_\sigma G(\sigma(\eta),\eta)^{-1} \partial_\eta G(\sigma(\eta),\eta)$ and the mean value theorem we obtain
    \begin{align*}
        |\sigma(\eta)-\sigma_0| &\leq \sup_{s \in [0,1]} |\partial_\eta \sigma(s\eta )| \eta
        \ls \sup_{s \in [0,1]} |\partial_\eta G(\sigma(s\eta ),s\eta)| \eta 
        \ls (\sup_{s \in [0,1]} |\sigma_1- \sigma(s\eta)| + \eps)  \eta \\
        &\ls \eta_0\sup_{s \in [0,1]} |\sigma(s\eta)-\sigma_0|  + \eta |\sigma_1-\sigma_0| + \eps \eta, 
    \end{align*}
    which yields $\sup_{s \in [0,1]} |\sigma(s\eta)-\sigma_0| \ls \eta (|\sigma_1-\sigma_0| + \eps)$, provided $\eta_0>0$ is sufficiently small. 
    Thus, we arrive at
    $$
        |\sigma(\eta)-\sigma_0| \leq \sup_{s \in [0,1]} |\sigma(s\eta) -\sigma_0| \ls \eta (|\sigma_1-\sigma_0|+ \eps).
    $$

    (ii) We define $G \colon \R^2 \to \R$ by $G(\varsigma,\eta) := \langle \bpsi_{\sigma_0} + \eta \vb - \bpsi_{\varsigma}, \bvphi_{\eps}^* \rangle_{L^2}$
    and find $G(\sigma_0,0) = 0$ as well as $\partial_\varsigma G(\sigma_0,0) = \langle \bpsi_{\sigma_0}',\bvphi_\eps^* \rangle_{L^2} = 1 + O(|\sigma_0-\sigma_*| + \eps) \not = 0$. 
    This yields the claim upon applying again the implicit function theorem.
\end{proof}

\subsection{Modulational dynamics: proof of Theorem~\ref{thm:main} (i)}\label{sec:ModulationDynamics}
Under the assumptions of Theorem~\ref{thm:main} (i) we construct a solution of the form $\ub(t) = \bpsi_{\sigma(t)} + \vb(t)$ for all $t \geq 0$. We prove that $\vb$ is of order $O(\eps)$ for all $t \geq 0$ provided $\sigma$ satisfies the ODE for the position~\eqref{eq:ODEsigma} and $\eps \ll 1$. This part of the proof follows the approach in~\cite{Promislow2002Renormalization} and employs the \emph{renormalization group method} (see also~\cite{Promislow2007Oscillatory,Promislow2005Thermally} for further applications of this method). The challenge in our setting is that we need to extend the arguments in~\cite{Promislow2002Renormalization} to equations with a linear perturbation that is unbounded as an operator acting in $H^1(\T)$.

Inserting the solution ansatz into~\eqref{eq:LLE_sys} we obtain
\begin{align}\label{eq:ds1}
    - \bpsi_{\sigma}' \dot\sigma + \dot\vb = \El_{\eps}(\bpsi_{\sigma}) \vb + \Non_{\bpsi_\sigma}(\vb) + \eps V \bpsi_{\sigma}',
\end{align}
where $\Non_\ub(\vb):= N(\ub+\vb) - N(\ub) - \partial_\ub N(\ub)\vb$ for $\ub,\vb \in \R^2$. 
To write the principal linear part as a time-independent operator, let us define for $\sigma_0,\sigma_1\in \R$ the $\eps$-independent, bounded operator 
$\Be(\sigma_1,\sigma_0) := \El_{\eps}(\bpsi_{\sigma_1})-\El_{\eps}(\bpsi_{\sigma_0}) \colon H^1(\T) \to H^1(\T)$.
Then, we rewrite~\eqref{eq:ds1} as
\begin{align}\label{eq:ds2}
    \dot\vb = \El_{\eps}(\bpsi_{\sigma_0}) \vb + \Be(\sigma,\sigma_0)\vb + \Non_{\bpsi_\sigma}(\vb) + \bpsi_{\sigma}' \dot\sigma + \eps V \bpsi_{\sigma}'.
\end{align}
If $\|\ub(0) - \bpsi_{\tilde\sigma}\|_{H^1} \leq C \eta < \eta_0$ for some $\tilde{\sigma} \in \R$ we find according to Lemma~\ref{lem:inital_decomposition} (i), 
$\sigma_0 \in \R$, with $\sigma_0 = \tilde{\sigma} + O(\eta)$ such that $\vb_0:= \ub(0) - \bpsi_{\sigma_0} \in X_{\sigma_0,\eps}^\textup{s}$ and $\|\vb_0\|_{H^1} \leq \eta$, 
provided $C \in (0,1)$ is small enough. 
We then equip~\eqref{eq:ds2} with the initial condition $(\vb(0),\sigma(0)) =(\vb_0,\sigma_0) \in X_{\sigma_0,\eps}^\textup{s} \times \R$. 
The value of $\eta$ will be determined later in the proof, but let us already note at this point that $\eta = O(\eps)$. 
Let us further set $\mathcal{G}(\sigma,\sigma_0,\vb) := \Be(\sigma,\sigma_0)\vb + \Non_{\bpsi_\sigma}(\vb) + \eps V \bpsi_{\sigma}'$. 
Then the mild formulation of~\eqref{eq:ds2} is given by
\begin{align}\label{eq:ds3}
\begin{split}
    \vb(t) &= \eu^{\El_{\eps}(\bpsi_{\sigma_0}) t} \vb_0 
    + \int_0^t \eu^{\El_{\eps}(\bpsi_{\sigma_0})(t-s)} \left(\mathcal{G}(\sigma(s),\sigma_0,\vb(s)) + \bpsi_{\sigma(s)}' \dot\sigma(s) \right) \de s, \\ \sigma(0) &= \sigma_0.
\end{split}
\end{align}
Projecting onto $X_{\sigma_0,\eps}^\textup{s}$ and $X_{\sigma_0,\eps}^\textup{c}$, respectively, we arrive at
\begin{align}\label{eq:ds4}
\begin{split}
    \vb(t) &= S_{\sigma_0,\eps}(t) \vb_0 
    + \int_0^t S_{\sigma_0,\eps}(t-s) \Pi_{\sigma_0,\eps}^\textup{s} \left(\mathcal{G}(\sigma(s),\sigma_0,\vb(s)) +\bpsi_{\sigma(s)}' \dot\sigma(s) \right)  \de s,\\ 
    \dot{\sigma}(t) \langle \bpsi_{\sigma(t)}' ,\bphi_{\eps,\sigma_0}^* \rangle_{L^2} &= -\langle \mathcal{G}(\sigma(t),\sigma_0,\vb(t)),\bphi_{\sigma_0,\eps}^* \rangle_{L^2}   \\
    \sigma(0) &= \sigma_0.
\end{split}
\end{align}
In order to control the solution of~\eqref{eq:ds4} we need the estimates collected in the following lemma.
\begin{Lemma}\label{lem:estimates1}
    For $\|\vb\|_{H^1},|\sigma_1-\sigma_0| \ls 1$ the following estimates hold.
    \begin{enumerate}
        \item $\|\Pi_{\sigma_0,\eps}^\textup{s}\mathcal{G}(\sigma_1,\sigma_0,\vb)\|_{H^1} \ls \eps + |\sigma_1-\sigma_0|\|\vb\|_{H^1} + \|\vb\|_{H^1}^2 $;
        \item $\|\Pi_{\sigma_0,\eps}^\textup{s}\bpsi_{\sigma_1}' \|_{H^1} \ls 1$;
        \item $|\langle \bpsi_{\sigma_1}',\bphi_{\sigma_0,\eps}^*\rangle_{L^2} - 1| \ls \eps + |\sigma_1-\sigma_0|$;
        \item $|\langle \mathcal{G}(\sigma_1,\sigma_0,\vb),\bphi_{\sigma_0,\eps}^* \rangle_{L^2} -\eps \mathcal{V}_\eff(\sigma_1)| 
        \ls \eps^2 + \eps |\sigma_1-\sigma_0| + |\sigma_1-\sigma_0|\|\vb\|_{H^1} + \|\vb\|_{H^1}^2$.
    \end{enumerate}
\end{Lemma}
\begin{proof}
    We have, using the mean value theorem,
    \begin{align*}
        \|\Pi_{\sigma_0,\eps}^\textup{s} \mathcal{B}(\sigma_1,\sigma_0) \vb\|_{H^1} 
        \ls \|\mathcal{B}(\sigma_1,\sigma_0) \vb\|_{H^1}
        \ls |\sigma_1-\sigma_0| \|\vb\|_{H^1}.
    \end{align*}
    Moreover,
    \begin{align*}
        \|\Pi_{\sigma_0,\eps}^\textup{s} \Non_{\bpsi_{\sigma_1}}(\vb)\|_{H^1} 
        \ls \|\Non_{\bpsi_{\sigma_1}}(\vb)\|_{H^1}
        \ls \|\vb\|_{H^1}^2,\qquad
        \|\Pi_{\sigma_0,\eps}^\textup{s} \eps V \bpsi_{\sigma_1}'\|_{H^1} \ls \eps,
    \end{align*}
    which together prove the first estimate. The second estimate follows directly, and for the third estimate we observe that by Lemma~\ref{lem:perturbation_eigenvalue}, 
    $$
        \|\bphi_{\sigma_1,0}^*-\bphi_{\sigma_0,\eps}^*\|_{L^2} \ls \eps + |\sigma_1-\sigma_0|.
    $$
    Thus, the Cauchy-Schwarz inequality yields
    $$
        |\langle \bpsi_{\sigma_1}',\bphi_{\sigma_0,\eps}^*\rangle_{L^2} - 1| = |\langle \bpsi_{\sigma_1}',\bphi_{\sigma_0,\eps}^*-\bphi_{\sigma_1,0}^*\rangle_{L^2}| \ls \eps + |\sigma_1-\sigma_0|.
    $$
    Finally, we compute
    \begin{align*}
        |\langle \mathcal{G}(\sigma_1,\sigma_0,\vb),\bphi_{\sigma_0,\eps}^* \rangle_{L^2} -\eps \mathcal{V}_\eff(\sigma_1)| &= |\langle \mathcal{B}(\sigma_1,\sigma_0) \vb + \Non_{\bpsi_{\sigma_1}}(\vb) + \eps V \bpsi_{\sigma_1}',\bphi_{\sigma_0,\eps}^* \rangle_{L^2} - \eps \langle V \bpsi_{\sigma_1}', \bphi_{\sigma_1}^* \rangle_{L^2}|\\
        & \ls \eps |\langle V \bpsi_{\sigma_1}',\bphi_{\sigma_0,\eps}^*-\bphi_{\sigma_1}^*  \rangle_{L^2} |
        + |\sigma_1-\sigma_0| \|\vb\|_{H^1} + \|\vb\|_{H^1}^2 \\
        & \ls \eps^2 + \eps |\sigma_1-\sigma_0| + |\sigma_1-\sigma_0|\|\vb\|_{H^1} + \|\vb\|_{H^1}^2.
    \end{align*}
    This completes the proof.
\end{proof}

Local well-posedness of~\eqref{eq:ds4} in $H^1(\T) \times \R$ is established in the next lemma
and follows from a standard fixed point argument.
\begin{Lemma}\label{lem:local_existence}
    For all $(\vb_0,\sigma_0) \in H^1(\T) \times \R$ there exists a unique maximal solution 
    $$(\vb,\sigma) \in C([0,T_\textup{max}), H^1(\T)) \times C^1([0,T_\textup{max}),\R)$$ of~\eqref{eq:ds4} satisfying the alternative:
    $$
        T_\textup{max} = \infty \qquad\text{or}\qquad\lim_{t \to T_\textup{max}} \|\vb(t)\|_{H^1} + |\sigma(t)| + \frac{1}{|\langle \bpsi_{\sigma(t)}' ,\bphi_{\eps,\sigma_0}^* \rangle_{L^2}|} = \infty.
    $$
\end{Lemma}
Using Lemma~\ref{lem:estimates1}, we may write the ODE for the position as
\begin{align}\label{eq:ODE_Pos2}
    \dot\sigma(t) = - \eps \mathcal{V}_\eff(\sigma(t)) + O(\eps^2+ \eps|\sigma(t)-\sigma_0| + |\sigma(t)-\sigma_0|\|\vb(t)\|_{H^1} + \|\vb(t)\|_{H^1}^2)
\end{align}
provided $\eps,|\sigma(t)-\sigma_0| \ll 1$.

\medskip

Now we are in the position to close a nonlinear stability estimate.
Let us fix $t_0\geq 0$ and assume that $\ub(t_0) = \vb(t_0) + \bpsi_{\sigma_0}$ with $\vb(t_0) \in X_{\sigma_0,\eps}^*$, $\|\vb(t_0)\|_{H^1} \leq \eta$. 
We define the template functions
$$
    \alpha_1(t) := \sup_{s \in [t_0,t]} \eu^{\kappa (s-t_0)} \|\vb(s)\|_{H^1}, \qquad
    \beta_1(t) := \sup_{s \in [t_0,t]} |\sigma(s) - \sigma_0|  
$$
for all $t \geq t_0$ for which the solution exists. Then, $\alpha_1,\beta_1$ are continuous and $\alpha_1(t_0) \leq \eta$, $\beta_1(t_0)=0$. 
As long as $\alpha_1(t),\beta_1(t)\ls 1$, we find, using~\eqref{eq:ds4} and Lemmas~\ref{lem:LinearDecay},~\ref{lem:estimates1},
\begin{align*}
    \|\vb(t)\|_{H^1} 
    &\ls \eu^{-\kappa (t-t_0)} \eta + \eu^{-\kappa (t-t_0)} \int_{t_0}^t \eu^{\kappa (s-t_0)}
    \left( \eps + \|\vb(s)\|_{H^1} |\sigma(s)-\sigma_0| + \|\vb(s)\|_{H^1}^2 \right) \de s  \\
    &\ls \eu^{-\kappa (t-t_0)} \eta + \eu^{-\kappa (t-t_0)} \int_{t_0}^t  \eu^{\kappa (s-t_0)} \eps + \alpha_1(t) \beta_1(t) + \eu^{-\kappa (s-t_0)} \alpha_1(t)^2\de s \\
    & \ls \eu^{-\kappa (t-t_0)} \eta + \eps + \eu^{-\kappa (t-t_0)} (t-t_0) \alpha_1(t) \beta_1(t) +  \eu^{-\kappa (t-t_0)}\alpha_1(t)^2 
\end{align*}
and thus, there exists $C>1$ such that
\begin{align*}
    \alpha_1(t) \leq C \left( \eta + \eps \eu^{\kappa(t-t_0)} + (t-t_0) \alpha_1(t) \beta_1(t) + \alpha_1(t)^2\right).
\end{align*}
Now we set $\eta = 6 C \eps$ and
\begin{align}\label{eq:def_t1}
\begin{split}
    t_1 := \sup\left\{t>t_0 :
    (t-t_0) \beta_1(t) \leq \frac{1}{2C}, \
    \kappa (t-t_0) \leq \log\left(\frac{2\eta}{\eps}\right) = \log\left( 12C\right), \
    \alpha_1(t) \leq 6 C \eta
    \right\}.
\end{split}
\end{align}
Then, for all $t \in [t_0,t_1]$ we find
\begin{align}\label{eq:estimate_T1}
    \alpha_1(t) \leq 2 C (2\eta + \alpha_1(t)^2) \leq 4 C \eta + 72 C^3 \eta^2 \leq 5C \eta
\end{align}
provided $\eps\ll 1$. Next, we show that for $\eps$ small, $t_1 = \kappa^{-1} \log(12C) +t_0$. 
Utilizing~\eqref{eq:ODE_Pos2} we find for $t \in [t_0,t_1]$,
\begin{align*}
    |\sigma(t) -\sigma_0| &\leq \int_{t_0}^t |\dot\sigma(s)| \de s
    \ls \int_{t_0}^t \eps + |\sigma(s) -\sigma_0| \|\vb (s)\|_{H^1} + \|\vb(s)\|_{H^1}^2 \de s \\
    &\ls \eps (t-t_0) + \alpha_1(t) \beta_1(t) + \alpha_1(t)^2,
\end{align*}
and thus there exists $C'>1$ such that $(t_1-t_0)\beta_1(t_1) \leq C' \eps$.
In particular, choosing $\eps>0$ smaller if necessary we infer $(t_1-t_0) \beta_1(t_1) \leq (C4)^{-1}$ which means, taking into account~\eqref{eq:estimate_T1},
$t_1 = \kappa^{-1} \log(12 C) + t_0$. 
Then, using again~\eqref{eq:estimate_T1} we obtain
$$
    \|\vb(t_1)\|_{H^1} \leq 6C \eta \eu^{-\kappa(t_1-t_0)} = \frac{\eta}{2} 
    \quad\text{and}\quad
    \ub(t_1) = \bpsi_{\sigma(t_1)} + \vb(t_1), \quad  \vb(t_1) \in X_{\sigma_0,\eps}^\textup{s}. 
$$
Combining the last estimate with Lemma~\ref{lem:inital_decomposition} (i), there exists $\sigma_1 \in \R$ such that
$$
    \vb_1:=\bpsi_{\sigma(t_1)} + \vb(t_1) - \bpsi_{\sigma_1} \in X_{\sigma_1,\eps}^\textup{s}, \quad
    |\sigma_1-\sigma(t_1)| \ls \eta |\sigma_0 -\sigma(t_1)| + \eta \eps
    \ls \eta^2.
$$
From this we infer, provided $\eps$ (and hence also $\eta$) is sufficiently small,
$$
    \|\vb_1\|_{H^1} \leq \|\vb(t_1)\|_{H^1} + \|\bpsi_{\sigma(t_1)} - \bpsi_{\sigma_1}\|_{H^1} 
    \leq \frac{\eta}{2} + |\sigma(t_1) -\sigma_1| 
    \leq \frac{\eta}{2} + C' \eta^2
    \leq \eta.
$$
Now we can iterate the argument with the updated initial condition $\vb(t_1) = \vb_1$, $\sigma(t_1) = \sigma_1$ at time $t_1 = t_0 + \kappa^{-1} \log(12C)$. 
Setting $t_0=0$ and $t_n:= n\kappa^{-1} \log(12C)$ for $n \in \N$ we have proved the following. 
There exists $\sigma \in C^1([0,\infty) \setminus \cup_{n \in \N} \{t_n\} )$ with discontinuities of size $O(\eta^2)$ and
\begin{align}\label{eq:estimate_on_u}
    \|\ub(t) - \bpsi_{\sigma(t)}\|_{H^1} \leq 6 C \eta \quad \text{for all }t\geq 0.
\end{align}
Moreover, from~\eqref{eq:ODE_Pos2} we find that $\sigma$ satisfies
\begin{align}\label{eq:ODE_Pos3}
    |\dot\sigma + \eps \mathcal{V}_\eff(\sigma)| \leq C \eps^2 \qquad
    \text{on }[0,\infty) \setminus  \bigcup_{n \in \N} \{t_n\},
\end{align}
upon choosing $C$ larger if necessary. Finally, following verbatim the arguments of~\cite[Sect.~2.2.2]{Promislow2002Renormalization} we can replace the piecewise smooth $\sigma$ by $\tilde{\sigma} \in C^1([0,\infty))$ such that 
$\|\sigma-\tilde\sigma\|_{L^\infty(0,\infty)} \ls \eta^2$ and~\eqref{eq:estimate_on_u} as well as~\eqref{eq:ODE_Pos3} continue to hold
for $\tilde\sigma$ (with $C$ replaced by $2C$). This proves the first part of Theorem~\ref{thm:main}.

\subsection{Asymptotic convergence: proof of Theorem~\ref{thm:main} (ii)}\label{sec:AsmyptoticConvergence}
Throughout this section, let the assumptions of Theorem~\ref{thm:main} (i) and (ii) be satisfied.
The first lemma establishes the initial transit of the solution to an $\eps$-neighborhood of the asymptotic solution realized by following the flow of the ODE~\eqref{eq:ODEsigma} for $\sigma$.

\begin{Lemma}\label{lem:InitialTransit}
    There exists $0 \leq t_1 \ls \eps^{-2}$ such that
    \begin{align*}
        |\sigma(t) - \sigma_*| \ls \eps \quad\text{and}\quad \|\ub(t) - \bpsi_{\sigma_*,\eps}\|_{H^1} \ls \eps \quad\text{for all }t \geq t_1.
    \end{align*}
\end{Lemma}
\begin{proof}
    We start with the proof of the $\sigma$-estimate. First note that $\mathcal{I}_{\sigma_*}$ is a positive invariant set for the ODE~\eqref{eq:ODEsigma} and hence $\sigma(t) \in \mathcal{I}_{\sigma_*}$ for all $t \geq 0$. Now, let $\sigma \in \mathcal{I}_{\sigma_*}$ with $|\mathcal{V}_\eff(\sigma)| \leq 2 C_1\eps$. By Taylor's theorem, there exists $C>0$ for which $|\mathcal{V}_\eff(\sigma) - (\sigma-\sigma_*) \mathcal{V}_\eff'(\sigma_*)| \leq C (\sigma-\sigma_*)^2$. If $|\sigma-\sigma_*| = 4C_1\eps/\mathcal{V}_\eff'(\sigma_*)$ we obtain $|\mathcal{V}_\eff(\sigma)|\geq 3C_1\eps$ provided $\eps$ is sufficiently small. Hence, we have established the inequality $|\sigma-\sigma_*| \leq 4 C_1 \eps/\mathcal{V}_\eff'(\sigma_*)$. In contrast, for all $\sigma \in \mathcal{I}_{\sigma_*}$ with $|\sigma-\sigma_*|> 4 C_1 \eps/\mathcal{V}_\eff'(\sigma_*)$ we find $|\mathcal{V}_\eff(\sigma)|> 2 C_1 \eps$. Combining this with the estimate $|\dot\sigma(t) + \eps \mathcal{V}_\eff(\sigma(t))| \leq C_1 \eps^2$ we find that there exists $t_1 \geq 0$ such that for all $t \geq t_1$ it holds $|\sigma(t) - \sigma_*| \leq 4 C_1\eps /\mathcal{V}_\eff'(\sigma_*)$. Moreover, integrating the ODE~\eqref{eq:ODEsigma} we obtain the upper bound $t_1 \ls \eps^{-2}$. 
    To establish the second estimate, we use the first estimate to find
    $$
        \|\ub(t) - \bpsi_{\sigma_*,\eps}\|_{H^1} 
        \leq \|\ub(t) - \bpsi_{\sigma(t)}\|_{H^1} + \|\bpsi_{\sigma(t)} - \bpsi_{\sigma_*}\|_{H^1} 
        + \|\bpsi_{\sigma_*} - \bpsi_{\sigma_*,\eps}\|_{H^1} \ls \eps
    $$
    for all $t \geq t_1$.
\end{proof}

\begin{Remark}\label{rem:FailedEstimate}
    It is tempting to use the estimate from Lemma~\ref{lem:InitialTransit} as a starting point for a bootstrap argument to establish the asymptotic stability. 
    However, this estimate is in general not strong enough to close such an argument. The reason is that the semigroup $S_\eps(t)$ generated by the linearization about $\bpsi_{\sigma_*,\eps}$ decays only at a slow rate $\eu^{-\eps\kappa t}$, cf.~Lemma~\ref{lem:LinearDecay}. 
    As a consequence, if we define $\tilde{\wb}(t) = \ub(t) - \bpsi_{\sigma_*,\eps}$
    and insert $\ub = \tilde{\wb} + \bpsi_{\sigma_*,\eps}$ into~\eqref{eq:LLE_sys} we obtain
    \begin{align*}
        \tilde{\wb} (t) = S_\eps(t - t_1) \tilde{\wb}(t_1) + \int_{t_1}^t S_\eps(t - s) \tilde{\Non}(\tilde{\wb}(s)) \de s, 
    \end{align*}
    where $\|\tilde{\Non}(\wb)\|_{H^1} \ls \|\wb\|_{H^1}^2$ and $\|\tilde{\wb}(t_1)\|_{H^1} \ls \eps$.
    For the template function $$\tilde{\alpha}(t) = \sup_{s \in [t_1,t]} \|\tilde{\wb}(s)\| \eu^{\eps \kappa(s-t_1)}$$ a direct computation then yields the estimate
    \begin{align*}
        \|\tilde{\wb}(t)\|_{H^1} &\ls \eps \eu^{-\eps \kappa(t-t_1)} + \int_{t_1}^t \eu^{-\eps\kappa(t-s)} \|\tilde{\wb}(s)\|_{H^1}^2 \de s 
        \ls \left(\eps + \eps^{-1} \tilde{\alpha}(t)^2 \right) \eu^{-\eps \kappa(t-t_1)}.
    \end{align*}
    From this we deduce $\tilde{\alpha}(t) \leq C( \eps  + \eps^{-1} \tilde{\alpha}(t)^2)$, which is not sufficient to close an estimate on $\tilde{\alpha}$, due to the presence of the $\eps^{-1}$ factor.
\end{Remark}

To resolve the problem outlined in Remark~\ref{rem:FailedEstimate}, we identify an intermediate time regime of size $-\eps^{-1} \log(\eps)$ on which we improve the error to $\|\ub(t)-\bpsi_{\sigma_*,\eps}\|_{H^1} \ls \eps^{3/2}$. In this argument, we establish improved estimates on the nonlinearity relying on parity arguments.
This is precisely the reason why we require the additional evenness assumption on $\bpsi_0$ for Theorem~\ref{thm:main} (ii). The new error estimate is proved in Section~\ref{sec:IntermediatRegime}.
Once this is achieved, we can apply standard bootstrap arguments in Section~\ref{sec:AsymptoticRegime} to prove nonlinear asymptotic convergence to the stationary solution $\bpsi_{\sigma_*,\eps}$.

\subsubsection{Intermediate time regime}\label{sec:IntermediatRegime}
Let $t_1\geq 0$ be given as in Lemma~\ref{lem:InitialTransit}.
By applying Lemma~\ref{lem:inital_decomposition} we find $\sigma_j \in \R$ with $|\sigma_j-\sigma_*|\ls \eps$, $j=1,2$ such that
$$
    \tilde\vb_1:=\ub(t_1) - \bpsi_{\sigma_1} \in  X_{\eps}^\textup{s}, \qquad 
    \wb_\eps:=\bpsi_{\sigma_*,\eps} - \bpsi_{\sigma_2} \in  X_\eps^\textup{s}.
$$
This decomposition suggests to write the solution as $\ub(t) = \vb(t) + \bpsi_{\varsigma(t)} -\bpsi_{\sigma_2} + \bpsi_{\sigma_*,\eps}$ 
with $\vb(t)= \ub(t) -\bpsi_{\varsigma(t)} -\wb_\eps$ for $t\geq t_1$. 
From this, we derive, similarly to Section~\ref{sec:ModulationDynamics}, the mild formulation
\begin{align}\label{eq:ds5}
\begin{split}
    \vb(t) &= S_\eps(t-t_1)\Pi_\eps^\textup{s} \vb_1 + \int_{t_1}^t S_\eps(t-s)\Pi_\eps^\textup{s} \Big(
    {\mathcal{H}}(\varsigma(s), \sigma_2, \vb(s)) + \dot{\varsigma}(s) \bpsi_{\varsigma(s)}' \Big) \de s \\
    \dot{\varsigma}(t) \langle \bpsi_{\varsigma(t)}', \bvphi_\eps^* \rangle_{L^2} &=  -\langle 
    \mathcal{H}(\varsigma(t) ,\sigma_2, \vb(t)), \bvphi_\eps^* \rangle_{L^2} , \\
    \varsigma(t_1) &= \sigma_1,
\end{split}
\end{align}
where we abbreviate 
$\mathcal{H}(\varsigma ,\sigma_2, \vb) = \El_\eps(\psi_{\sigma_*,\eps}) (\bpsi_{\varsigma} - \bpsi_{\sigma_2}) 
+ \Non_{\bpsi_{\sigma_*,\eps}}(\vb + \bpsi_{\varsigma} - \bpsi_{\sigma_2})$ 
and $\vb_1 = \vb(t_1)$. By the same reasoning as in Lemma~\ref{lem:local_existence} there exists a unique local solution to~\eqref{eq:ds5}. In the next lemma we derive estimates that we need to obtain bounds on $\vb$ and $\varsigma-\sigma_2$.
\begin{Lemma}\label{lem:estimates2}
    For $\|\vb\|_{H^1},|\varsigma-\sigma_2|\ls 1$ the following estimates hold.
    \begin{enumerate}
        \item $\|\vb_1\|_{H^1} \ls \eps$;
        \item $\|\Pi_\eps^\textup{s} \mathcal{H}(\varsigma ,\sigma_2, \vb)\|_{H^1} \ls \eps  |\varsigma - \sigma_2| + |\varsigma - \sigma_2|^2 + \|\vb\|_{H^1}^2$;
        \item $\|\Pi_\eps^\textup{s} \bpsi_{\varsigma}'\|_{H^1} \ls |\varsigma - \sigma_2| + \eps$;
        \item
        $
            \begin{array}{ll}
                |\langle \mathcal{H}(\varsigma ,\sigma_2, \vb), \bvphi_\eps^* \rangle_{L^2} - \lambda_1\eps (\varsigma-\sigma_2)|\\
                \qquad\qquad\qquad\qquad\ls \eps^2 |\varsigma - \sigma_2| + \eps |\varsigma - \sigma_2|^2 
                + |\varsigma-\sigma_2| \|\vb\|_{H^1}+ \|\vb\|_{H^1}^2 + |\varsigma-\sigma_2|^3;
            \end{array}
        $
        \item $|\langle \bpsi_{\varsigma}', \bvphi_\eps^* \rangle_{L^2}- 1| \ls |\varsigma - \sigma_2| + \eps$.
    \end{enumerate}
\end{Lemma}
\begin{proof}
    To prove the first estimate, we observe with $\vb_1 = \tilde{\vb}_1 - \wb_\eps$ that
    \begin{align*}
        \|\vb_1\|_{H^1}
        &\leq \|\tilde\vb_1\|_{H^1} + \|\wb_\eps\|_{H^1} 
        \leq \|\ub(t_1)-\bpsi_{\sigma(t_1)}\|_{H^1} + \|\bpsi_{\sigma(t_1)} - \bpsi_{\sigma_1}\|_{H^1} + \eps
        \ls \eps.
    \end{align*}

    For the second estimate, we first compute 
    $\Pi_\eps^\textup{s} (\bpsi_{\varsigma} - \bpsi_{\sigma_2}) = (\bpsi_{\varsigma} - \bpsi_{\sigma_2}) 
    - \langle\bpsi_{\varsigma} - \bpsi_{\sigma_2},\bvphi_\eps^*\rangle_{L^2}  \bvphi_\eps$ and the identities
    \begin{align*}
        \bpsi_{\varsigma} - \bpsi_{\sigma_2} &= 
        (\sigma_2-\varsigma) \bpsi_{\xi \varsigma + (1-\xi) \sigma_2}'
        = (\sigma_2-\varsigma) \bvphi_\eps + O (|\sigma_2-\varsigma|( |\varsigma-\sigma_*| + |\sigma_2-\sigma_*| + \eps) )    ,\\
        \langle\bpsi_{\varsigma} - \bpsi_{\sigma_2},\bvphi_\eps^*\rangle_{L^2}
        &= \langle \bpsi_{\xi \varsigma + (1-\xi) \sigma_2}' , \bvphi_\eps^* \rangle_{L^2} (\sigma_2 - \varsigma) 
        = (1 + O( |\varsigma-\sigma_*| + |\sigma_2-\sigma_*|+ \eps) )  (\sigma_2-\varsigma),
    \end{align*}
    for some $\xi \in (0,1)$.
    Therefore, we find 
    $\Pi_\eps^\textup{s} (\bpsi_{\varsigma} - \bpsi_{\sigma_2}) 
    = O(|\varsigma - \sigma_2| (|\varsigma-\sigma_*| + |\sigma_2-\sigma_*|+\eps))$ in $H^3(\T)$, where we implicitly used $\|\bpsi_0\|_{H^4} \ls 1$ and $\|\bvphi_\eps - \bpsi_{\sigma_*}'\|_{H^3} \ls \eps$. We then find
    \begin{align*}
        \|\Pi_\eps^\textup{s}\El_\eps(\bpsi_{\sigma_*,\eps})(\bpsi_{\varsigma} - \bpsi_{\sigma_2})\|_{H^1} \ls
        \|\Pi_\eps^\textup{s}(\bpsi_{\varsigma} - \bpsi_{\sigma_2}) \|_{H^3} 
        \ls  |\varsigma - \sigma_2| (|\varsigma-\sigma_*| + |\sigma_2-\sigma_*| + \eps).
    \end{align*}
    Furthermore, we have 
    \begin{align*}
        \|\Pi_\eps^\textup{s} \Non_{\bpsi_{\sigma_*,\eps}}(\vb + \bpsi_{\varsigma} - \bpsi_{\sigma_2})\|_{H^1}
        \ls \|\Non_{\bpsi_{\sigma_*,\eps}}(\vb + \bpsi_{\varsigma} - \bpsi_{\sigma_2})\|_{H^1} 
        \ls |\varsigma - \sigma_2|^2 + \|\vb\|_{H^1}^2,
    \end{align*}
    which proves the second inequality, upon noting $|\sigma_2-\sigma_*|\ls \eps$. 

    For the third estimate, we compute
    \begin{align}\label{eq:proofestimates2}
        \|\bpsi_{\varsigma}' - \bvphi_\eps\|_{H^1} = 
        \|\bpsi_{\varsigma}' -\bpsi_{\sigma_*}' + \bpsi_{\sigma_*}' -  \bvphi_\eps \|_{H^1}
        = O(|\varsigma - \sigma_2| + \eps)
    \end{align}
    which proves the assertion since $\Pi_\eps^\textup{s}\bpsi_{\varsigma}' = \Pi_\eps^\textup{s}(\bpsi_{\varsigma}'-\bvphi_\eps)$.

    For the fourth estimate we first observe
    \begin{align*}
        \langle \El_\eps(\bpsi_{\sigma_*,\eps}) (\bpsi_{\varsigma} - \bpsi_{\sigma_2}), \bvphi_\eps^*\rangle_{L^2} &=
        \langle \bpsi_{\varsigma} - \bpsi_{\sigma_2}, \El_\eps(\bpsi_{\sigma_*,\eps})^*\bvphi_\eps^*\rangle_{L^2}
        = \lambda_\eps (\sigma_2-\varsigma) \langle\bpsi_{\xi \varsigma + (1-\xi) \sigma_2}'  ,\bvphi_\eps^*\rangle_{L^2}\\
        &= -\lambda_1\eps (\sigma_2-\varsigma) (1 + O(|\varsigma -  \sigma_2| + \eps).
    \end{align*}
    In order to obtain the improved estimate on the nonlinear part, we expand the nonlinearity up to second order. 
    To this end, we introduce the shorthand notation $\Delta\bpsi = \bpsi_{\varsigma} - \bpsi_{\sigma_2}$, $\Delta\sigma =\varsigma -\sigma_2$. 
    Taylor's formula then yields
    \begin{align*}
        \Non_{\bpsi_{\sigma_*,\eps}}(\vb + \Delta\bpsi) = \frac{1}{2}\partial_\ub^2N(\bpsi_{\sigma_*,\eps})[\vb+\Delta\bpsi]^2 + O(\|\vb + \Delta\bpsi\|_{H^1}^3).
    \end{align*}
    We expand $\partial_\ub^2N(\bpsi_{\sigma_*,\eps}) = \partial_\ub^2N(\bpsi_{\sigma_*}) + O(\eps)$ and for some $\xi \in (0,1)$,
    $$
        [\vb + \Delta\bpsi]^2 = [\vb - \Delta\sigma \bpsi_{\sigma_2 + \xi \Delta\sigma}']^2
        = [\vb]^2 - \Delta\sigma ([\vb, \bpsi_{\sigma_2 + \xi \Delta\sigma}'] + [ \bpsi_{\sigma_2 + \xi \Delta\sigma}',\vb] )
        + (\Delta\sigma)^2 [\bpsi_{\sigma_2 + \xi \Delta\sigma}']^2.
    $$
    The last term can be written as $[\bpsi_{\sigma_2 + \xi \Delta\sigma}']^2 = [\bpsi_{\sigma_*}']^2 + O(\eps + |\Delta\sigma|)$.
    Combining the last four equalities yields
    \begin{align*}
        &\langle \Non_{\bpsi_{\sigma_*,\eps}}(\vb + \bpsi_{\varsigma} - \bpsi_{\sigma_2}), \bvphi_\eps^*\rangle_{L^2}\\
        &\qquad= \frac{1}{2} (\Delta\sigma)^2\langle \partial_\ub^2 N(\bpsi_{\sigma_*}) [\bpsi_{\sigma_*}']^2, \bphi_{\sigma_*}^* \rangle_{L^2} 
        + O(\eps |\Delta\sigma|^2 +|\Delta\sigma| \|\vb\|_{H^1} + \|\vb\|_{H^1}^2 + |\Delta\sigma|^3).
    \end{align*}
    Finally, we observe that 
    $\langle \partial_\ub^2 N(\bpsi_{\sigma_*}) [\bpsi_{\sigma_*}']^2, \bphi_{\sigma_*}^* \rangle_{L^2} 
    = \langle \partial_\ub^2 N(\bpsi_{0}) [\bpsi_{0}']^2, \bphi_{0}^* \rangle_{L^2} = 0$
    and the last equality holds since $\partial_\ub^2 N(\bpsi_{0})[\bpsi_0']^2$ is an even function and $\bphi_0^*$ is odd, see~\cite[Lem.~7]{Bengel2024Pinning}.
    This yields the desired estimate.

    The final estimate follows directly from identity~\eqref{eq:proofestimates2}.
\end{proof}

Now we define the template functions
$$
    \alpha_2(t) = \sup_{s \in [t_1,t]} \|\vb(s)\|_{H^1} \eu^{\eps \kappa (s-t_1)}, \qquad
    \beta_2(t) = \sup_{s \in [t_1,t]} |\varsigma(s)-\sigma_2|, \qquad\text{for }t\geq t_1.
$$
For $\beta_2$ it is possible to deduce an a priori bound, which we then use to improve the error up to $\eps^{3/2}$.
\begin{Lemma}\label{lem:AprioriBeta}
    The a priori estimate $\beta_2(t) \ls \eps$ holds for all $t\geq t_1$.
\end{Lemma}
\begin{proof}
    For $t\geq t_1$, we find by Theorem~\ref{thm:main} (i) and Lemma~\ref{lem:InitialTransit}, $\|\ub(t)-\bpsi_{\sigma(t)}\|_{H^1} \ls \eps$
    and $|\sigma(t) -\sigma_*| \ls \eps$. Applying Lemma~\ref{lem:inital_decomposition} (ii) point-wise for every $t \geq t_1$ there exists a unique $\tilde{\varsigma}(t)$ with $|\tilde{\varsigma}(t) -\sigma(t)| \ls \eps$ such that 
    $\ub(t) - \bpsi_{\tilde\varsigma(t)} \in  X_{\eps}^\textup{s}$. At the same time, the solution to~\eqref{eq:ds5} yields a decomposition with $\ub(t)- \bpsi_{\varsigma(t)} \in X_\eps^\textup{s}$ for $t \geq t_1$. Therefore, we deduce $\varsigma=\tilde\varsigma$ on $[t_1,\infty)$ and, thus, infer $|\varsigma(t) - \sigma_2| \ls |\varsigma(t) - \sigma(t)| +  |\sigma(t)-\sigma_2| \ls \eps$ for $t \geq t_1$.
    This yields the assertion upon taking the supremum over all $t \geq t_1$.
\end{proof}
Next, we aim for a detailed estimate on $\alpha_2$ and $\beta_2$.
By Lemma~\ref{lem:estimates2} we can write the equation for $\varsigma$ as
\begin{align}\label{eq:ODE_Pos4}
\begin{split}
    \dot{\varsigma}(t) &= - \eps \lambda_1 (\varsigma(t)- \sigma_2) + \mathcal{R}(\varsigma(t)-\sigma_2, \vb),
\end{split}
\end{align}
with $|\mathcal{R}(\varsigma, \vb)| \ls  |\varsigma|\|\vb\|_{H^1}  +  \eps^2 |\varsigma| + \eps |\varsigma|^2 + \|\vb\|_{H^1}^2+ |\varsigma|^3$. From~\eqref{eq:ds5},~\eqref{eq:ODE_Pos4}, and Lemmas~\ref{lem:LinearDecay},~\ref{lem:estimates2} we find
\begin{align*}
    \|\vb(t)\|_{H^1} &\ls \eu^{-\kappa (t-t_1)} \|\vb_1\|_{H^1} 
    + \int_{t_1}^t \eu^{-\kappa (t-s)} \left( \eps \beta_2(t) + \beta_2(t)^2 + \|\vb(s)\|_{H^1}^2 \right) \de s \\
    &\ls \eu^{-\kappa (t-t_1)} \eps + \eu^{-\kappa t} \left( \eps \beta_2(t) \eu^{\kappa t} + \beta_2(t)^2 \eu^{\kappa t} + \eu^{\kappa t-2\kappa \eps(t-t_1)} \alpha_2(t)^2 \right) \\
    & \ls \eu^{-\kappa (t-t_1)} \eps +  \eps \beta_2(t) + \beta_2(t)^2 + \eu^{-2\kappa \eps (t-t_1)} \alpha_2(t)^2.
\end{align*}
This implies
\begin{align}\label{eq:EstimateAlpha2}
    \alpha_2(t) \ls  \eu^{- \frac{\kappa}{2} (t-t_1)}\eps + \eu^{\kappa \eps (t-t_1)} (\eps \beta_2(t) + \beta_2(t)^2) + \alpha_2(t)^2
    \leq C \left( \eu^{- \frac{\kappa}{2} (t-t_1)}\eps + \eu^{\kappa \eps (t-t_1)} \eps^2 + \alpha_2(t)^2\right). 
\end{align}
Now we proceed by deriving an estimate for $\beta_2$. Integrating~\eqref{eq:ODE_Pos4}, we find
\begin{align*}
    \varsigma(t) - \sigma_2 = \eu^{- \eps \lambda_1 (t-t_1)} (\sigma_1-\sigma_2) + \int_{t_1}^t \eu^{-\eps\lambda_1(t-s)} \mathcal{R}(\varsigma(s) -\sigma_2,\vb(s)) \de s,
\end{align*}
and hence using Lemma~\ref{lem:AprioriBeta},
\begin{align*}
\begin{split}
    \beta_2(t) &\ls \eu^{-\eps\lambda_1(t-t_1)} |\sigma_1-\sigma_2| \\
    & \qquad+ \int_{t_1}^t \eu^{-\eps\lambda_1 (t-s)} \left( \beta_2(t) \|\vb(s)\|_{H^1} + \eps^2 \beta_2(t) + \eps \beta_2(t)^2 + \|\vb(s)\|_{H^1}^2 + \beta_2(t)^3 \right) \de s \\
    & \ls \eu^{-\eps\lambda_1 (t-t_1)} \eps 
    + \eu^{-\eps \lambda_1 t} \int_{t_1}^t \eu^{\eps \lambda_1 t_1}\alpha_2(t) \beta_2(t) + \eu^{\eps\lambda_1 s} \eps^3 + \eu^{-\eps\lambda_1 s + 2\eps\kappa t_1} \alpha_2(t)^2 + \eu^{\eps\lambda_1 s} \eps^3 \de s \\
    &  \ls  \eu^{-\eps\kappa(t-t_1)} \eps + \eu^{-\eps\kappa (t-t_1)} (t-t_1) \alpha_2(t) \beta_2(t) + \eps^2 + \eps^{-1} \eu^{-\eps\kappa(t-t_1)} \alpha_2(t)^2,
\end{split}
\end{align*}
upon taking $\kappa>0$ smaller if necessary.
Since $\eu^{-\eps\kappa (t-t_1)} (t-t_1) \ls \eps^{-1}$ we arrive at
\begin{align}\label{eq:EstimateBeta2}
    \beta_2(t) \leq C' \left(\eu^{-\eps\kappa(t-t_1)} \eps + \eps^{-1} \alpha_2(t) \beta_2(t) + \eps^2 + \eps^{-1} \eu^{-\eps\kappa(t-t_1)} \alpha_2(t)^2 \right).
\end{align}
Now we define the intermediate time-scale by setting $t_2 - t_1 = - \frac{\log(\eps)}{2\eps\kappa}> 0$. Then, for $t \in [t_1,t_2]$ we have in~\eqref{eq:EstimateAlpha2}
\begin{align*}
    \alpha_2(t) \leq C \left( \eu^{-\eps\kappa (t-t_1)} \eps + \eps^{3/2} + \alpha_2(t)^2 \right).
\end{align*}
Let $\tau_1 := \sup\{t \in [t_1,t_2] : \alpha_2(t) \leq 3 C \eps\}$. Then, for all $t \in [t_1,\tau_1)$ we have
\begin{align*}
    \alpha_2(t) \leq C \left(\eps + \eps^{3/2} +9 C^2\eps^2 \right) \leq 2 C \eps 
\end{align*}
provided that $\eps$ is sufficiently small. This proves $\tau_1 = t_2$ and 
$$
    \alpha_2(t_2) \ls \eu^{-\eps\kappa(t_2-t_1)} \eps + \eps^{3/2} + \eps^2 \ls \eps^{3/2}.
$$
Inserting the last estimate into~\eqref{eq:EstimateBeta2} we then find
$$
    \beta_2(t_2) \ls \eu^{-\eps\kappa(t_2-t_1)} \eps + \eps^{-1} \alpha_2(t_2) \beta_2(t_2) + \eps^2 + \eps^{-1} \eu^{-\eps\kappa(t_2-t_1)} \alpha_2(t_2)^2 
    \ls \eps^{3/2} + \eps^{3/2} + \eps^2 + \eps^{2} \ls \eps^{3/2}.
$$
In total, we have thus established the improved estimate
\begin{align}\label{eq:EstimateAfterIntermediate}
    \|\ub(t_2) - \bpsi_{\sigma_*,\eps}\|_{H^1} \leq \|\vb(t_2)\|_{H^1} + \|\bpsi_{\varsigma(t_2)} - \bpsi_{\sigma_2}\|_{H^1} 
    \ls \alpha_2(t_2) + \beta_2(t_2) \ls \eps^{3/2}.
\end{align}

\subsubsection{Asymptotic time regime}\label{sec:AsymptoticRegime}
For the asymptotic convergence, we study the solution, after transitioning through the intermediate time regime, on $[t_2,\infty)$.
Here $t_2 = t_1 - \log(\eps)/(2\kappa\eps) $ with $t_1$ given in Lemma~\ref{lem:InitialTransit}.
We set $\wb(t):= \ub(t) - \bpsi_{\sigma_*,\eps}$ for $t \geq t_2$. 
Then, $\wb$ satisfies the equation
\begin{align*}
    \wb(t) &= S_\eps(t-t_2) \wb(t_2) + \int_{t_2}^t S_\eps(t-s) \Non_{\bpsi_{\sigma_*,\eps}}(\wb(s)) \de s,
\end{align*}
for $t \geq t_2$, where we recall from Lemma~\ref{lem:LinearDecay} that $\|S_\eps(t)\|_{H^1 \to H^1} \ls \eu^{-\eps\kappa t}$ as well as $\|\Non_{\bpsi_{\sigma_*,\eps}}(\vb)\|_{H^1} \ls \|\vb\|_{H^1}^2$. We define 
$$
    \alpha_3(t) := \sup_{s \in [t_2,t]} \|\wb(s)\|_{H^1} \eu^{\eps\kappa (s-t_2)}, \qquad t \geq t_2.
$$
Then, since $\|\wb(t_2)\|_{H^1} = \|\ub(t_2) - \bpsi_{\sigma_*,\eps}\|_{H^1} \ls \eps^{3/2}$, we obtain for $t\geq t_2$,
\begin{align*}
    \|\wb(t)\|_{H^1} &\ls \eu^{-\eps\kappa (t-t_2)} \|\wb(t_2)\|_{H^1} + \int_{t_2}^t \eu^{-\eps\kappa (t-s)} \|\wb(s)\|_{H^1}^2 \de s \\
    &\ls \eu^{-\eps\kappa (t-t_2)} \eps^{3/2} + \eps^{-1} \eu^{-\eps\kappa (t-t_2)} \alpha_3(t)^2,
\end{align*}
which yields
\begin{align}
    \alpha_3(t) \leq C \left( \eps^{3/2} + \eps^{-1}  \alpha_3(t)^2\right).
\end{align}
This motivates the definition $\tau :=\sup\{t \geq t_2 : \alpha_3(t) \leq 3 C \eps^{3/2}\}$. Clearly, $\tau>t_2$ and for all $t \in[t_2,\tau)$ we find
\begin{align*}
    \alpha_3(t) \leq C \eps^{3/2} +  9 C^4 \eps^{2} \leq 2 C \eps^{3/2}
\end{align*}
for $\eps$ sufficiently small. 
Thus $\tau= \infty$ and the asymptotic stability $\|\ub(t) - \bpsi_{\sigma_*,\eps}\|_{H^1} \ls \eu^{-\eps\kappa (t-t_2)}$ for $t \geq t_2$ is proved. Since $t_2 = t_1 - \log(\eps)/(2\kappa\eps) \ls \eps^{-2} - \log(\eps)/(2\kappa\eps)$, we find
$
    \eu^{\eps\kappa t_2} \ls \eu^{K\eps^{-1}}\eps^{-1},
$
$K>0$, which yields the desired estimate
$$
    \|\ub(t) - \bpsi_{\sigma_*,\eps}\|_{H^1} \ls
    \frac{\eu^{K\eps^{-1}}}{\eps}\eu^{-\eps\kappa t} \qquad\text{for all $t \geq 0$.}
$$
In summary, we have thus established part (ii) of Theorem~\ref{thm:main}.

\appendix

\section{An implicit function theorem with uniform bounds}\label{sec:Append_IFT}

Throughout this paper, we need a version of the implicit function theorem which provides uniform bounds on the implicitly defined function.
Although we expect this result to be well-known, we could not find a reference in the literature and thus state it here for the sake of completeness.
\begin{Theorem}\label{thm:IFT}
Let $X,Y$ be Banach spaces and $F=F(x,\lambda) \in C^2(X\times \R,Y)$. Assume that there exists $(x_0,\lambda_0) \in X \times \R$ and $\eps,C>0$ such that
\begin{itemize}
    \item $F(x_0,\lambda_0) = 0$;
    \item $\partial_x F(x_0,\lambda_0) \colon X \to Y$ is bijective;
    \item for all $(x,\lambda) \in \overline{B_\eps(x_0,\lambda_0)}$:
    $$
        \|(\partial_xF(x,\lambda))^{-1}\|_{Y \to X} + \sum_{j=1,2} \|\partial_x^j F(x,\lambda)\|_{X^j \to Y} + \|\partial_x\partial_\lambda F(x,\lambda)\|_{X \times \R \to Y} + \|\partial_\lambda^j F(x,\lambda)\|_{\R^2 \to Y} \leq C.
    $$
\end{itemize}
Then, there exist $\delta,M>0$ depending only on $\eps, C$ and a $C^1$-branch $B_\delta(\lambda_0) \ni \lambda \to x(\lambda) \in X$ such that $(x(\lambda),\lambda) \in X \times \R$ is the unique solution of
$$
    F(x,\lambda) = 0
$$
in a small neighborhood of $(x_0,\lambda_0)$. Moreover, we have $\|x(\lambda)-x_0\| \leq M |\lambda-\lambda_0|$ for all $\lambda \in B_\delta(\lambda_0)$.
\end{Theorem}
\begin{proof}
The proof follows line by line the proof of the implicit function theorem~\cite[Thm.~2.3]{Ambrosetti1995Primer}.
However, in order to obtain uniform bounds, we employ Taylor's formula~\cite[p.~28]{Ambrosetti1995Primer} up to second order, to control the size of constants that appear in the contraction mapping argument.
\end{proof}

\section*{Acknowledgments} Funded by the Deutsche Forschungsgemeinschaft (DFG, German Research Foundation) -- Project-ID 258734477 -- SFB 1173. 

\bibliographystyle{plain}
\bibliography{bibliography}

\end{document}